\documentclass[11pt]{amsart}

\usepackage{
  amssymb,amsmath,txfonts,mathrsfs,
  mathtools,titletoc,tikz
}
\usepackage{color}

\definecolor{reviewgreen}{rgb}{0,0.50,0}

\usepackage{graphicx}
\usepackage{float}
\usepackage{enumitem}

\usetikzlibrary{
  arrows.meta,
  positioning,
  calc,
  decorations.pathreplacing
}

\usepackage{hyperref}
\usepackage[nameinlink,capitalise,noabbrev]{cleveref}

\allowdisplaybreaks
\hypersetup{
  colorlinks=true,
  linkcolor=blue,
  citecolor=blue,
  urlcolor=blue
}

\theoremstyle{plain}
\newtheorem{theorem}{Theorem}[section]
\newtheorem{proposition}[theorem]{Proposition}
\newtheorem{lemma}[theorem]{Lemma}
\newtheorem{corollary}[theorem]{Corollary}

\theoremstyle{plain}

\newtheorem{definition}[theorem]{Definition}
\newtheorem{notation}[theorem]{Notation}

\theoremstyle{plain}
\newtheorem{remark}[theorem]{Remark}

\newcommand{\Hn}{\mathbb H^n}

\newcommand{\bdry}{\partial_\infty}
\newcommand{\Jac}{\operatorname{Jac}}
\newcommand{\Vol}{\operatorname{Vol}}

\newcommand{\dist}{\operatorname{dist}}
\newcommand{\tr}{\operatorname{tr}}
\newcommand{\Id}{\operatorname{Id}}

\newcommand{\dd}{\,\mathrm{d}}
\newcommand{\HS}{\mathrm{HS}}

\DeclareMathOperator{\tension}{\tau}

\newcommand{\R}{\mathbb R}

\numberwithin{equation}{section}
\begin{document}
\title[A Harmonic-Map Proof of Mostow Rigidity]
{A Harmonic-Map Proof of Mostow Rigidity}
\author{Guoyi Xu}
\address{Department of Mathematical Sciences\\Tsinghua University, Beijing\\P. R. China}
\email{guoyixu@tsinghua.edu.cn}
\date{\today}
\begin{abstract}
Given an isomorphism between the fundamental groups of two compact
hyperbolic manifolds of dimension at least three, we consider the
associated equivariant harmonic lift.  We establish a sharp pointwise gradient estimate for this lift. This yields a harmonic-map proof of Mostow
rigidity.
\end{abstract}

\thanks{Guoyi Xu was partially supported by NSFC 12141103.}
\maketitle

\tableofcontents

% ----------------------------------------------------------------
% Compile-visible record of the preamble-only typography change.
% Remove this box after reviewing the marked copy.
% ----------------------------------------------------------------

% ================================================================
\section{Introduction}\label{sec:statement}
% ================================================================

Throughout, \(M^n\) and \(N^n\) are compact hyperbolic
\(n\)-manifolds with \(n\geq3\).  Set
\(\Gamma:=\pi_1(M^n)\) and \(\Lambda:=\pi_1(N^n)\), and let
\(\rho:\Gamma\longrightarrow\Lambda\) be an isomorphism.
Since \(\Hn\) is contractible, compact hyperbolic manifolds are
\(K(\pi,1)\)-spaces, so \(\rho\) is realized by a homotopy equivalence
\(M\to N\); see \cite[Proposition~1B.9]{Hatcher}.

Mostow proved in 1968 that compact real hyperbolic manifolds of dimension
at least three are determined, up to isometry, by their fundamental
groups, using the quasiconformal geometry of the sphere at infinity
\cite{Mostow1968}.  Prasad subsequently extended strong rigidity to the
noncompact finite-volume setting \cite{Prasad1973}.  A substantially
different proof was later obtained by Besson--Courtois--Gallot through
their minimal-entropy method \cite{BessonCourtoisGallot}.

By the Eells--Sampson theorem \cite[\S11A]{EellsSampson}, the corresponding homotopy class contains a harmonic map
\[
\bar u:M\longrightarrow N.
\]
The Eells--Sampson construction above suggests another natural route:
prove directly that the harmonic representative of the prescribed
homotopy equivalence is an isometry.

After choosing lifts, we obtain a harmonic map $u:\Hn\longrightarrow\Hn$ satisfying
\begin{equation}\label{eq:equivariance}
u(\gamma x)=\rho(\gamma)u(x),
\qquad
\gamma\in\Gamma,
\quad x\in\Hn.
\end{equation}

Li--Tam established fundamental
uniqueness and boundary-regularity results for proper harmonic maps
between hyperbolic spaces \cite{LiTam,LiTamII}, while Li--Wang proved
uniqueness for harmonic rough isometries into negatively curved Hadamard
spaces \cite{LiWang}.  On the existence side of the Schoen--Li--Wang
program, Lemm--Markovi\'c treated real hyperbolic spaces, and
Benoist--Hulin established the corresponding theorem for rank-one
symmetric spaces \cite{LemmMarkovic,BenoistHulin}.  These works provide the closest analytic framework for a harmonic-map
proof, but they do not by themselves establish the remaining rigidity
step.

To the author's knowledge, no previous argument has begun with the
equivariant harmonic representative and proved directly that it is an
isometry without invoking Mostow rigidity itself or an essentially
equivalent boundary-conformal rigidity theorem.  This is consistent with
the recent survey of Daskalopoulos--Mese, which records that a harmonic-map
proof of Mostow's theorem was still missing
\cite[p.~104]{DaskalopoulosMese}.  

For a smooth map \(F:(X,g)\to(Y,h)\), define
\begin{equation}\label{eq:HS_norm_def}
|dF_x|^2:=\operatorname{tr}_g(F^*h)_x= \sum_{i=1}^{\dim X}|dF_x(e_i)|_h^2,
\end{equation}
where \(e_1,\dots,e_{\dim X}\) form any \(g\)-orthonormal basis of \(T_xX\).  Thus \(|dF_x|\) is the Hilbert--Schmidt norm of the
differential.  Some references denote the same quantity by
\(|\nabla F|\); throughout this paper,
\[
|\nabla F|^2=|dF|^2.
\]
The energy density is $e(F)=\frac12|dF|^2$.

The main theorem below addresses
precisely the remaining rigidity step through a sharp pointwise energy
identity.

\begin{theorem}\label{thm:main}
Let \(M^n\) and \(N^n\) be connected compact hyperbolic
\(n\)-manifolds without boundary, where \(n\geq3\), and let
\[
\rho:\pi_1(M)\longrightarrow\pi_1(N)
\]
be an isomorphism.  Let
\[
\bar u:M\longrightarrow N
\]
be the Eells--Sampson harmonic representative of the corresponding
homotopy class, and let
\[
u:\mathbb H^n\longrightarrow\mathbb H^n
\]
be its \(\rho\)-equivariant lift.  Then
\[
|du|^2\equiv n
\qquad\text{on }\mathbb H^n.
\]
\end{theorem}

\begin{corollary}[Mostow rigidity]\label{cor:mostow}
Under the hypotheses of \Cref{thm:main}, the harmonic representative
\[
\bar u:M\longrightarrow N
\]
is a global Riemannian isometry.  
\end{corollary}

The principal analytic obstruction to this harmonic-map program is the
absence, in the real-hyperbolic setting, of a coercive differential
identity that forces the harmonic representative to be an isometry.

Indeed, if \(u\) is harmonic between hyperbolic
\(n\)-manifolds and \(Q:=u^*h\), then the ordinary Bochner formula becomes
\begin{align}
\frac12\Delta |du|^2
=
|\nabla du|^2
-(n-1)|du|^2
+|du|^4-|Q|^2.
\label{eq:intro_hyperbolic_bochner}
\end{align}
If \(\lambda_1,\dots,\lambda_n\) are the singular values of \(du\), the
zero-order part of the right-hand side is
\begin{align}
-(n-1)\sum_{i=1}^n\lambda_i^2
+2\sum_{1\leq i<j\leq n}\lambda_i^2\lambda_j^2,
\nonumber
\end{align}
which has no fixed sign.  Thus integration of
\eqref{eq:intro_hyperbolic_bochner} does not force
\(\nabla du=0\).

Nor do harmonicity and rough-isometry control alone
provide the missing equality: rough isometry is a coarse condition and
allows nonconformal boundary behavior.  

The general existence and
uniqueness theory for harmonic maps associated with quasi-isometries
illustrates the breadth of this class; see Benoist--Hulin
\cite[Theorem~1.1 and Fact~1.4]{BenoistHulin}.

%The present argument obtains the missing rigidity by exploiting the cocompact lattice equivariance in \eqref{eq:equivariance}.

%Consequently, the cocompact lattice equivariance in
%\eqref{eq:equivariance} must enter essentially.

%The present proof uses that symmetry in a different way from the averaging-and-conservation-law direction suggested by Daskalopoulos--Mese.    

We firstly use equivariance and cocompactness to
show that \(u\) is a globally Lipschitz rough isometry and that
\(|du|^2\) is \(\Gamma\)-invariant.  The boundary extension and the
elementary regularity argument then produce one point
\(\xi\in\partial_\infty\Hn\) at which the boundary map has an invertible
differential
\begin{align}
A:=D(\bdry u)_\xi.
\nonumber
\end{align}

After a fixed source--target boundary-coordinate normalization at
\(\xi\), and writing \(A\) for the differential of the boundary map in
the resulting charts, we consider the scale-dependent maps
\begin{align}
u_s
&:=
\delta_{1/s}\circ u\circ\delta_s,
&
\delta_s(x,y)&:=(sx,sy).
\nonumber
\end{align}
Although \(\delta_s\) is a Euclidean dilation in upper half-space
coordinates, it is an exact hyperbolic isometry.  Thus no hyperbolic
metric is rescaled and no Euclidean interior tangent space appears.
Intrinsically, \(u_s\) is a pointed recentering at infinity, prescribed by
the ordinary Euclidean tangent-map blow-up of the boundary map:
\begin{align}
\bdry u_s(z)
=
\frac{(\bdry u)(sz)}s
\longrightarrow
Az.
\nonumber
\end{align}
We therefore call \(\{u_s\}\) the boundary-tangent renormalization of
\(u\).  Every \(C^\infty_{\mathrm{loc}}\) subsequential limit of
\(u_s\) is identified with the explicit harmonic map
\begin{align}
H_A(x,y)
=
\left(
Ax,
\sqrt{\frac{\operatorname{tr}(A^*A)}{n-1}}\,y
\right),
\qquad
|dH_A|^2\equiv n.
\nonumber
\end{align}

Notice that \(H_A\) need not be an isometry: the point is that its
Hilbert--Schmidt energy has the exact value \(n\), independently of the
anisotropy of \(A\).

The delicate step is to turn this single boundary differential into an
identity for the original map.  Merely obtaining convergence at the single
recentered basepoint would not suffice.  The basepoint estimate, elliptic
compactness, identification of the limiting boundary map, and harmonic
rough-isometry uniqueness together show that every subsequential limit of
the renormalized family equals \(H_A\).  

This yields full
\(C^\infty_{\mathrm{loc}}\) convergence and hence, for every fixed
\(R<\infty\), the moving-ball estimate
\begin{align}
\lim_{t\to\infty}
\sup_{z\in B(r_\xi(t),R)}
\left||du_z|^2-n\right|
=0.
\nonumber
\end{align}

The deterministic cocompactness argument in
\Cref{lem:cocompact_propagation} translates a compact fundamental set into
these moving balls.  Since \(|du|^2\) is \(\Gamma\)-invariant, the
moving-ball limit propagates to
\begin{align}
|du|^2\equiv n
\qquad\text{on }\Hn.
\nonumber
\end{align}

Thus the role played in the suggested program by averaging is replaced
here by boundary-tangent renormalization, harmonic uniqueness, and
cocompact propagation.

The inverse homotopy class, the degree formula, and equality in the
Jacobian--energy inequality then convert this sharp trace identity into
the statement that \(\bar u\) is an isometry. 

In this sense, the full moving-ball convergence is the mechanism that overcomes the
noncoercivity of the ordinary Bochner formula in the present proof.

The most direct historical precursor is Mostow's original
quasiconformal proof.  At a boundary point where the quasiconformal
conjugacy is differentiable with invertible differential, Mostow considers
renormalized boundary maps and combines the resulting tangent map with the
dynamics of hyperbolic dilations and a conformal-capacity argument to force
that tangent map to be conformal; see
\cite[pp.~99--101]{Mostow1968}.  The present proof instead passes the
possibly anisotropic tangent through the nonlinear harmonic extension
problem and extracts the sharp trace identity from the resulting harmonic
model.

The analytically closest predecessors are the boundary regularity theory of
Li--Tam for proper harmonic maps between hyperbolic spaces
\cite{LiTam,LiTamII} and the normal-scale calculation of
Akutagawa--Matsumoto for asymptotically hyperbolic manifolds
\cite[Lemma~1.3 and Corollary~1.4]{AkutagawaMatsumoto}.  Those results assume
\(C^1\) control on a boundary neighborhood.  Here differentiability at one
boundary point is combined with harmonic rough-isometry uniqueness and
cocompact propagation.  A detailed comparison is given after
Proposition~\ref{prop:tangent}.

Acknowledgements: The author thanks Peter Li for first bringing this problem to his
attention and for teaching a course on harmonic maps during the author's
postdoctoral appointment at the University of California, Irvine.  The
author is deeply grateful to Jiaping Wang for pointing out that the
central issue is to obtain a sharp gradient estimate for the equivariant
harmonic map.  Finally, the author remembers with deep gratitude the late
Robert Gulliver, his doctoral advisor, who once told him, ``What matters
in mathematics is proof.''  These words often returned to him as he
sought the proofs presented in this paper.

% ================================================================
\section{The harmonic lift is a rough isometry}
%\ReviewUnusedLabel{sec:rough}
% ================================================================

The first substantial role of group equivariance is coarse geometric.

\begin{definition}\label{def:rough-isometry}
We call a map \(F:X\to Y\) a \emph{rough isometry} if there are constants
\(L\geq1\), \(C\geq0\), and \(C_0\geq0\) such that
\begin{equation}\label{eq:general-rough-isometry-bounds}
L^{-1}d_X(x,x')-C
\leq
d_Y(F(x),F(x'))
\leq
L\,d_X(x,x')+C
\end{equation}
for all \(x,x'\in X\), and every point of \(Y\) lies within distance
\(C_0\) of \(F(X)\).
\end{definition}

We now verify that the equivariant harmonic lift has this coarse-geometric
property.

\begin{proposition}\label{prop:rough}
The equivariant harmonic lift \(u:\Hn\to\Hn\) is a globally Lipschitz rough
isometry.
\end{proposition}

\begin{proof}
\textbf{Step (1)}. Let
\begin{align}
\pi_M:\mathbb H^n\longrightarrow M,
\qquad
\pi_N:\mathbb H^n\longrightarrow N
\nonumber
\end{align}
be the universal covering projections.  

Fix \(x_0\in\mathbb H^n\), and set
\begin{align}
p_0:=\pi_M(x_0),
\qquad
y_0:=u(x_0),
\qquad
q_0:=\pi_N(y_0)=\bar u(p_0).
\nonumber
\end{align}

Thus \(\bar u\) is regarded as a based map
\begin{align}
\bar u:(M,p_0)\longrightarrow(N,q_0),
\nonumber
\end{align}
and, under the identifications of the deck-transformation groups with
the corresponding based fundamental groups determined by \(x_0\) and
\(y_0\), its induced homomorphism is precisely
\(\rho:\Gamma\to\Lambda\).

Since \(M\) and \(N\) are \(K(\pi,1)\)-spaces and \(\rho\) is an
isomorphism, the inverse homomorphism \(\rho^{-1}\) is represented by a
based map
\begin{align}
\bar v:(N,q_0)\longrightarrow(M,p_0).
\nonumber
\end{align}
After a smooth approximation relative to \(q_0\), we may assume that
\(\bar v\) is smooth and still induces \(\rho^{-1}\).  

The two based maps \(\bar v\circ\bar u\) and \(\Id_M\) induce the same homomorphism
of \(\pi_1(M,p_0)\), and similarly \(\bar u\circ\bar v\) and
\(\Id_N\) induce the same homomorphism of \(\pi_1(N,q_0)\).

The based classification of maps into \(K(\pi,1)\)-spaces therefore gives based homotopies
\begin{align}
H_M&:M\times[0,1]\longrightarrow M,
&
H_M(\cdot,0)&=\bar v\circ\bar u,
&
H_M(\cdot,1)&=\Id_M,
\nonumber\\
H_N&:N\times[0,1]\longrightarrow N,
&
H_N(\cdot,0)&=\bar u\circ\bar v,
&
H_N(\cdot,1)&=\Id_N,
\nonumber
\end{align}
such that
\begin{align}
H_M(p_0,t)=p_0,
\qquad
H_N(q_0,t)=q_0
\qquad\text{for every }t\in[0,1].
\nonumber
\end{align}
For the realization and based-homotopy classification used here, see
\cite[Proposition~1B.9, pp.~90--91]{Hatcher}.

The previously fixed map \(u\) is the unique lift of \(\bar u\) satisfying
\(u(x_0)=y_0\).  Let
\begin{align}
\widetilde v:\mathbb H^n\longrightarrow\mathbb H^n
\nonumber
\end{align}
be the unique lift of \(\bar v\) satisfying
\(\widetilde v(y_0)=x_0\).  Since \(\bar v_*=\rho^{-1}\), the two lifts
satisfy
\begin{align}
u(\gamma x)&=\rho(\gamma)u(x),
&
\widetilde v(\lambda y)&=\rho^{-1}(\lambda)\widetilde v(y)
\nonumber
\end{align}
for every \(\gamma\in\Gamma\), \(\lambda\in\Lambda\), and
\(x,y\in\mathbb H^n\).  Consequently,
\begin{align}
\widetilde v\circ u(\gamma x)
&=\gamma\,\widetilde v\circ u(x),
&
u\circ\widetilde v(\lambda y)
&=\lambda\,u\circ\widetilde v(y).
\nonumber
\end{align}

Lift \(H_M\) uniquely to
\begin{align}
\widetilde H_M:\mathbb H^n\times[0,1]\longrightarrow\mathbb H^n
\nonumber
\end{align}
by requiring
\begin{align}
\widetilde H_M(x,0)=\widetilde v\circ u(x).
\nonumber
\end{align}
Because \(H_M(p_0,t)=p_0\) and
\(\widetilde v\circ u(x_0)=x_0\), the path
\(t\mapsto\widetilde H_M(x_0,t)\) is constantly equal to \(x_0\).
Hence \(\widetilde H_M(\cdot,1)\) is the lift of \(\Id_M\) fixing
\(x_0\), and therefore
\begin{align}
\widetilde H_M(x,1)=x
\qquad\text{for every }x\in\mathbb H^n.
\nonumber
\end{align}
The uniqueness of the lifted homotopy also gives
\begin{align}
\widetilde H_M(\gamma x,t)
=\gamma\widetilde H_M(x,t)
\qquad
\text{for every }\gamma\in\Gamma.
\nonumber
\end{align}

Similarly, lift \(H_N\) by requiring
\begin{align}
\widetilde H_N(y,0)=u\circ\widetilde v(y).
\nonumber
\end{align}
The based condition implies
\begin{align}
\widetilde H_N(y,1)&=y,
&
\widetilde H_N(\lambda y,t)
&=\lambda\widetilde H_N(y,t)
\nonumber
\end{align}
for every \(y\in\mathbb H^n\), \(\lambda\in\Lambda\), and
\(t\in[0,1]\).

Define
\begin{align}
F_M(x,t)&:=d\bigl(\widetilde H_M(x,t),x\bigr),
&
F_N(y,t)&:=d\bigl(\widetilde H_N(y,t),y\bigr).
\nonumber
\end{align}
The maps \(F_M\) and \(F_N\) are continuous because the lifted
homotopies and the hyperbolic distance function are continuous.

We first verify their invariance explicitly.  Every deck
transformation of \(\mathbb H^n\) is an isometry.  Hence, using the
equivariance of the lifted homotopies, for every
\(\gamma\in\Gamma\), \(x\in\mathbb H^n\), and \(t\in[0,1]\), we have
\begin{align}
F_M(\gamma x,t)&=d\bigl(\widetilde H_M(\gamma x,t),\gamma x\bigr)=d\bigl(\gamma\widetilde H_M(x,t),\gamma x\bigr)
=d\bigl(\widetilde H_M(x,t),x\bigr)=F_M(x,t).
\nonumber
\end{align}

Similarly, for every \(\lambda\in\Lambda\),
\(y\in\mathbb H^n\), and \(t\in[0,1]\),
\begin{align}
F_N(\lambda y,t)&=d\bigl(\widetilde H_N(\lambda y,t),\lambda y\bigr)=d\bigl(\lambda\widetilde H_N(y,t),\lambda y\bigr)=d\bigl(\widetilde H_N(y,t),y\bigr)
=F_N(y,t).
\nonumber
\end{align}

Thus \(F_M\) is constant on every fiber of
\(\pi_M\times\Id_{[0,1]}\), while \(F_N\) is constant on every fiber
of \(\pi_N\times\Id_{[0,1]}\).  Consequently there are well-defined
functions
\begin{align}
\overline F_M:M\times[0,1]&\longrightarrow[0,\infty),
&
\overline F_N:N\times[0,1]&\longrightarrow[0,\infty)
\nonumber
\end{align}
characterized by
\begin{align}
\overline F_M\bigl(\pi_M(x),t\bigr)
&=F_M(x,t)
=d\bigl(\widetilde H_M(x,t),x\bigr),
\nonumber\\
\overline F_N\bigl(\pi_N(y),t\bigr)
&=F_N(y,t)
=d\bigl(\widetilde H_N(y,t),y\bigr).
\nonumber
\end{align}
These functions are continuous.  

Since \(M\) and \(N\) are compact and \([0,1]\) is compact, the
spaces \(M\times[0,1]\) and \(N\times[0,1]\) are compact.  By the
extreme-value theorem, the finite constants
\begin{align}
b_M&:=\max_{(p,t)\in M\times[0,1]}\overline F_M(p,t)<\infty,
\nonumber\\
b_N&:=\max_{(q,t)\in N\times[0,1]}\overline F_N(q,t)<\infty
\nonumber
\end{align}
are well defined.  Pulling these bounds back to the universal
covers gives
\begin{align}
d\bigl(\widetilde H_M(x,t),x\bigr)&\leq b_M
&&\text{for every }(x,t)\in\mathbb H^n\times[0,1],
\nonumber\\
d\bigl(\widetilde H_N(y,t),y\bigr)&\leq b_N
&&\text{for every }(y,t)\in\mathbb H^n\times[0,1].
\nonumber
\end{align}
Finally, by the prescribed initial values of the lifted homotopies,
\begin{align}
\widetilde H_M(x,0)&=\widetilde v\circ u(x),
&
\widetilde H_N(y,0)&=u\circ\widetilde v(y).
\nonumber
\end{align}
Evaluating the preceding uniform bounds at \(t=0\), we obtain
\begin{align}
d\bigl(\widetilde v\circ u(x),x\bigr)&\leq b_M
&&\text{for every }x\in\mathbb H^n,
\label{eq:tilde-v-and-u-bounds}\\
d\bigl(u\circ\widetilde v(y),y\bigr)&\leq b_N
&&\text{for every }y\in\mathbb H^n.
\label{eq:lifted-homotopy-bounds}
\end{align}

%Thurston proves that lifts of homotopy equivalences between closed hyperbolic
%manifolds are pseudo-isometries in
%\cite[\S~5.9, Lemma~5.9.1 and its complete proof, p.~105]{ThurstonGT3M}.
%We now apply precisely the estimate in that proof.

Since \(\bar u\) and \(\bar v\) are smooth maps between compact Riemannian manifolds, their lifts are globally Lipschitz.  Let \(L_u,L_{\widetilde v}\geq1\) be Lipschitz constants for \(u\) and \(\widetilde v\), respectively. 

For \(x,x'\in\mathbb H^n\), from (\ref{eq:tilde-v-and-u-bounds}) and the above Lipschitz property, we obtain
\begin{align}
d(x,x')-2b_M
&\leq
d\bigl(\widetilde v\circ u(x),
       \widetilde v\circ u(x')\bigr)
\leq
L_{\widetilde v}\,d\bigl(u(x),u(x')\bigr).
\nonumber
\end{align}

It follows that
\begin{align}
L_{\widetilde v}^{-1}d(x,x')-\frac{2b_M}{L_{\widetilde v}}
\leq d\bigl(u(x),u(x')\bigr)
\leq L_u\,d(x,x').
\label{eq:Thurston-applied-to-u}
\end{align}
Thus \(u\) is a pseudo-isometry in Thurston's sense.

\textbf{Step (2)}. It remains only to verify the coarse-surjectivity condition. For every \(y\in\mathbb H^n\),
the second inequality in \eqref{eq:lifted-homotopy-bounds} gives
\begin{align}
d\bigl(y,u(\widetilde v(y))\bigr)\leq b_N.
\nonumber
\end{align}
Therefore every point of the target lies within distance \(b_N\) of
\(u(\mathbb H^n)\).  
Taking
\begin{align}
L:=\max\{L_u,L_{\widetilde v}\},
\qquad
C:=\frac{2b_M}{L_{\widetilde v}},
\qquad
C_0:=b_N,
\nonumber
\end{align}
and combining the coarse-surjectivity estimate with
\eqref{eq:Thurston-applied-to-u}, we obtain all the conditions in
Definition~\ref{def:rough-isometry}.

Hence \(u\) is a globally Lipschitz rough isometry.
\end{proof}

% ================================================================
\section{Boundary extension of a rough isometry}
%\ReviewUnusedLabel{sec:boundary}

\begin{definition}\label{def:quasisymmetry}
%\ReviewLabelChange{def:power-quasisym}{def:quasisymmetry}
Let \((X,d_X)\) and \((Y,d_Y)\) be metric spaces, and let
\begin{align}
\eta:[0,\infty)\longrightarrow[0,\infty)
\nonumber
\end{align}
be an increasing homeomorphism. A homeomorphism
\begin{align}
f:X\longrightarrow Y
\nonumber
\end{align}
is called \textbf{\(\eta\)-quasisymmetric} if
\begin{align}
\frac{d_Y\bigl(f(x),f(a)\bigr)}
     {d_Y\bigl(f(x),f(b)\bigr)}
\leq
\eta\!\left(
\frac{d_X(x,a)}{d_X(x,b)}
\right)
\label{eq:metric-quasisymmetry}
\end{align}
for all distinct \(x,a,b\in X\).

The map \(f\) is called \textbf{quasisymmetric} if it is
\(\eta\)-quasisymmetric for some increasing homeomorphism
\(\eta:[0,\infty)\to[0,\infty)\).
\end{definition}

% ================================================================

\begin{definition}[Ideal boundary]\label{def:ideal-boundary}
The ideal boundary \(\bdry\Hn\) is the set of equivalence classes of
unit-speed geodesic rays
\begin{align}
c:[0,\infty)\longrightarrow\Hn,
\nonumber
\end{align}
where two rays \(c\) and \(c'\) are equivalent if they are asymptotic,
that is, if
\begin{align}
\sup_{t\geq0}d\bigl(c(t),c'(t)\bigr)<\infty.
\nonumber
\end{align}
Equivalently, their images have finite Hausdorff distance.  The
equivalence class of \(c\) is denoted by \(c(\infty)\).
\end{definition}

For every \(o\in\Hn\) and every \(\xi\in\bdry\Hn\), there is a unique
unit-speed geodesic ray
\begin{align}
r_{o,\xi}:[0,\infty)\longrightarrow\Hn
\nonumber
\end{align}
such that
\begin{align}
r_{o,\xi}(0)=o,
\qquad
r_{o,\xi}(\infty)=\xi.
\nonumber
\end{align}
When the base point \(o\) is fixed, we write \(r_\xi:=r_{o,\xi}\).
These facts are standard; see Bridson--Haefliger
\cite[Definition~II.8.1 and Proposition~II.8.2]{BridsonHaefliger}.

\begin{definition}\label{def:Gromov-product}
Fix \(o\in\Hn\).  For \(a,b\in\Hn\), put
\begin{align}
(a\mid b)_o
:=
\frac12\bigl(d(o,a)+d(o,b)-d(a,b)\bigr).\nonumber
\end{align}

Using the ray \(r_\xi=r_{o,\xi}\) fixed above, for \(z\in\Hn\) and
\(\xi\in\bdry\Hn\) define the mixed Gromov product by

\begin{align}
(z\mid\xi)_o:=\lim_{t\to\infty}(z\mid r_\xi(t))_o=\lim_{t\to\infty}\frac12\bigl(d(o,z)+t-d(z,r_\xi(t))\bigr).
\label{eq:mixed-Gromov-product-definition}
\end{align}
By symmetry, set $(\xi\mid z)_o:=(z\mid\xi)_o$.  For \(\xi,\zeta\in\bdry\Hn\), define the boundary Gromov product by
\begin{align}
(\xi\mid\zeta)_o:=\lim_{t\to\infty}(r_\xi(t)\mid r_\zeta(t))_o=\lim_{t\to\infty}\frac12\bigl(2t-d(r_\xi(t),r_\zeta(t))\bigr). \label{eq:boundary-Gromov-product-definition}
\end{align}
The last limit is allowed to take the value \(+\infty\); in particular,
\((\xi\mid\xi)_o=+\infty\).  When the base point \(o\) is replaced by
\(x\in\Hn\), the products \((z\mid\xi)_x\) and
\((\xi\mid\zeta)_x\) are defined by the same formulas using the
unit-speed geodesic rays starting from \(x\).
\end{definition}

Let
\begin{align}
\mathbb B^n
:=
\left\{
x\in\mathbb R^n:|x|<1
\right\},
\qquad
g_{\mathbb B}
:=
\frac{4|dx|^2}{(1-|x|^2)^2}
\nonumber
\end{align}
be the Poincar\'e ball model. Choose an isometry
\begin{align}
I_o:\Hn\longrightarrow(\mathbb B^n,g_{\mathbb B})
\nonumber
\end{align}
such that \(I_o(o)=0\). For every \(\xi\in\bdry\Hn\), the curve
\(I_o\circ r_\xi\) is a radial geodesic, so there is a unique
\(\iota_o(\xi)\in\mathbb S^{n-1}\) such that
\begin{align}
I_o(r_\xi(t))
=
\tanh\left(\frac t2\right)\iota_o(\xi),
\qquad
t\geq0.
\label{eq:radial-ball-identification}
\end{align}
The resulting homeomorphism
defined by \eqref{eq:radial-ball-identification}
\begin{align}
\iota_o:\bdry\Hn\longrightarrow\mathbb S^{n-1}
\nonumber
\end{align}
is called a ball-model boundary identification based at \(o\).

\begin{lemma}[Ball-model visual metric formula]
\label{lem:elementary-hyperbolic-boundary-estimates}
The mixed and boundary Gromov products in
Definition~\ref{def:Gromov-product}, defined by
\eqref{eq:mixed-Gromov-product-definition} and
\eqref{eq:boundary-Gromov-product-definition}, are well defined. Identify
\(\Hn\cup\bdry\Hn\) with the closed Poincar\'e ball
\(\overline{\mathbb B^n}\subset\mathbb R^n\) so that
\(o=0\).  Throughout this lemma, \(|\cdot|\) denotes the
Euclidean norm on \(\mathbb R^n\).  If \(z\in\Hn\) and
\(\xi\in\bdry\Hn=\mathbb S^{n-1}\), then
\begin{align}
e^{-(z\mid\xi)_o}
=
\frac{|z-\xi|}{1+|z|}.
\label{eq:mixed-visual-ball-formula}
\end{align}
Consequently, for \(\xi,\zeta\in\bdry\Hn\),
\begin{align}
e^{-(\xi\mid\zeta)_o}
=
\frac12|\xi-\zeta|.
\label{eq:visual-is-chordal}
\end{align}
\end{lemma}

\begin{proof}
The hyperbolic distance and Busemann formulas in the Poincar\'e ball are
\begin{align}
d(o,z)
&=
\log\frac{1+|z|}{1-|z|},
&
\beta_\xi(z,o)
&=
\log\frac{|z-\xi|^2}{1-|z|^2}.
\nonumber
\end{align}
The second identity follows by substituting the radial point
\(r_\xi(t)\) into the hyperbolic distance formula and letting
\(t\to\infty\).  Since
\begin{align}
(z\mid\xi)_o
=
\frac12\bigl(d(o,z)-\beta_\xi(z,o)\bigr),
\nonumber
\end{align}
we obtain
\begin{align}
e^{-2(z\mid\xi)_o}
&=
\frac{1-|z|}{1+|z|}
\cdot
\frac{|z-\xi|^2}{1-|z|^2}
=
\frac{|z-\xi|^2}{(1+|z|)^2},
\nonumber
\end{align}
which proves \eqref{eq:mixed-visual-ball-formula}.  Letting
\(z\to\zeta\in\mathbb S^{n-1}\) proves
\eqref{eq:visual-is-chordal}.  The convergence criteria follow
immediately, and compactness of the closed-ball compactification is the
ordinary compactness of the Euclidean closed unit ball.
\end{proof}

\begin{definition}\label{def:visual-metric-ball-identification}
The visual metric based at \(o\) is
\begin{align}
\varrho_o(\xi,\zeta)
:=
e^{-(\xi\mid\zeta)_o},
\qquad
\xi,\zeta\in\bdry\Hn.
\label{eq:visual-metric-definition}
\end{align}
\end{definition}

Equip \(\mathbb S^{n-1}\) with its unit round metric and corresponding
geodesic distance
\begin{align}
d_{\mathbb S^{n-1}}(x,y)
:=
\arccos\langle x,y\rangle_{\mathbb R^n}.
\label{eq:round-distance-definition}
\end{align}
For the distance defined in
\eqref{eq:round-distance-definition}, \eqref{eq:visual-is-chordal} and the identity
\begin{align}
|x-y|
=
2\sin\left(\frac12d_{\mathbb S^{n-1}}(x,y)\right)
\nonumber
\end{align}
give
\begin{align}
\varrho_o(\xi,\zeta)
&=\frac12
\left|
\iota_o(\xi)-\iota_o(\zeta)
\right|=\sin\left(
\frac12 d_{\mathbb S^{n-1}}
\bigl(\iota_o(\xi),\iota_o(\zeta)\bigr)
\right).
\label{eq:visual-round-exact-relation}
\end{align}
Consequently, \eqref{eq:visual-round-exact-relation} implies
\begin{align}
2\varrho_o(\xi,\zeta)
\leq
d_{\mathbb S^{n-1}}
\bigl(\iota_o(\xi),\iota_o(\zeta)\bigr)
\leq
\pi\varrho_o(\xi,\zeta).
\label{eq:visual-round-bilipschitz}
\end{align}

\begin{definition}[Quasi-geodesic]
\label{def:quasi-geodesic}
Let \(I\subset\mathbb R\) be an interval, let \((X,d)\) be a metric
space, and let \(L\geq1\), \(C\geq0\).  A map
\begin{align}
q:I\longrightarrow X
\nonumber
\end{align}
is called an \emph{\((L,C)\)-quasi-geodesic} if
\begin{align}
L^{-1}|t-t'|-C
\leq d\bigl(q(t),q(t')\bigr)
\leq L|t-t'|+C
\qquad\text{for all }t,t'\in I.
\label{eq:quasi-geodesic-definition}
\end{align}
When \(I=[a,b]\), \(I=[0,\infty)\), or \(I=\mathbb R\), respectively,
we call \(q\) a quasi-geodesic segment, ray, or complete line.
\end{definition}

\begin{lemma}[Morse lemma for real hyperbolic space]
\label{lem:Morse}
Let \(d\geq2\), \(L\geq1\), and \(C\geq0\).  There exists a constant
\begin{align}
M=M(L,C)<\infty
\nonumber
\end{align}
with the following properties.
\begin{enumerate}[label=\textup{(\roman*)},leftmargin=2.8em]
\item Suppose that \(q:[0,\infty)\to\mathbb H^d\) is continuous and
satisfies \eqref{eq:quasi-geodesic-definition} for all
\(t,t'\geq0\).  Then there is a unique point
\(\xi\in\partial_\infty\mathbb H^d\) such that
\begin{align}
q(t)\longrightarrow\xi
\qquad(t\to\infty).
\nonumber
\end{align}
If \(\gamma:[0,\infty)\to\mathbb H^d\) is the geodesic ray satisfying
\(\gamma(0)=q(0)\) and \(\gamma(\infty)=\xi\), then
\begin{align}
d_{\mathrm H}\bigl(q([0,\infty)),\gamma([0,\infty))\bigr)\leq M.
\nonumber
\end{align}

\item Suppose that \(q:\mathbb R\to\mathbb H^d\) is continuous and
satisfies \eqref{eq:quasi-geodesic-definition} for all
\(t,t'\in\mathbb R\).  Then there are two distinct points
\(\xi^-,\xi^+\in\partial_\infty\mathbb H^d\) such that
\begin{align}
q(t)&\longrightarrow\xi^+
\quad(t\to+\infty),
&
q(t)&\longrightarrow\xi^-
\quad(t\to-\infty).
\nonumber
\end{align}
If \((\xi^-\xi^+)\) denotes the complete geodesic with ideal endpoints
\(\xi^-\) and \(\xi^+\), then
\begin{align}
d_{\mathrm H}\bigl(q(\mathbb R),(\xi^-\xi^+)\bigr)\leq M.
\label{eq:Morse-line-conclusion}
\end{align}
\end{enumerate}
Here \(d_{\mathrm H}\) denotes the Hausdorff distance between subsets of
\(\mathbb H^d\).  The constant \(M\) depends only on \(L\) and \(C\), and
not on \(q\), its endpoints, or \(d\).
\end{lemma}

\begin{proof}
The ray statement, including the existence and uniqueness of the ideal endpoint,
is \cite[Lemma~III.H.3.1, pp.~427--428]{BridsonHaefliger}.

For the line statement, apply the standard segment stability theorem
\cite[Theorem~III.H.1.7, pp.~401--402]{BridsonHaefliger} directly to
\(q|_{[-T,T]}\).

The resulting comparison
segments all have the same uniform Hausdorff constant \(M\).  Since
\(\mathbb H^d\) is proper, a diagonal Arzel\`a--Ascoli argument produces,
as \(T\to\infty\), a complete geodesic whose Hausdorff distance from
\(q(\mathbb R)\) is at most \(M\). 

Applying the ray statement to the two
half-lines of \(q\) shows that their ideal endpoints exist, are distinct,
and are precisely the two endpoints of this complete geodesic.
\end{proof}

\begin{proposition}\label{prop:rough-boundary-extension}
Let
$
F:\Hn\longrightarrow\Hn
$ 
be a continuous rough isometry.  Fix rough-isometry constants
\(L\geq1\), \(C\geq0\), and \(C_0\geq0\) as in
Definition~\ref{def:rough-isometry}.  Then \(F\) has a unique continuous
extension
\begin{align}
\overline F:\Hn\cup\bdry\Hn
\longrightarrow
\Hn\cup\bdry\Hn,
\nonumber
\end{align}
and
$
\bdry F:=\overline F|_{\bdry\Hn}
$ 
is a homeomorphism. Furthermore, for every \(o\in\Hn\), with \(o'=F(o)\), the map
\begin{align}
\bdry F:
(\bdry\Hn,\varrho_o)
\longrightarrow
(\bdry\Hn,\varrho_{o'})
\nonumber
\end{align}
is quasisymmetric, where the visual metrics are defined by
\eqref{eq:visual-metric-definition}.
\end{proposition}

\begin{proof}
Real hyperbolic space is Gromov hyperbolic and geodesic, hence almost
geodesic.  In the terminology of Bonk--Schramm \cite{BonkSchramm}, the two inequalities in
\eqref{eq:general-rough-isometry-bounds} say that \(F\) is an
\((L,C)\)-rough quasi-isometric embedding.  Since \(F(\Hn)\) is
\(C_0\)-cobounded, \(F\) is a rough quasi-isometry in their sense; for
this qualitative assertion one may replace their single additive
parameter by \(\max\{C,C_0\}\).

Bonk--Schramm
\cite[Proposition~6.3\textup{(1)}, \textup{(2)}, and \textup{(4)}]{BonkSchramm}
therefore give a well-defined bijection
\begin{align}
\bdry F:\bdry\Hn\longrightarrow\bdry\Hn.
\nonumber
\end{align}

Their Theorem~6.5\textup{(2)} implies that this bijection satisfies the
quasisymmetry inequality \eqref{eq:metric-quasisymmetry} for metrics in the
canonical boundary gauges;
see \cite[pp.~283--284]{BonkSchramm}.  In fact, their theorem gives the
stronger power-type control, but that additional quantitative information
will not be needed here.  Bonk--Schramm also observe that both this bijection and its inverse are continuous; see
\cite[p.~281]{BonkSchramm}.  Hence \(\bdry F\) is a homeomorphism and is
quasisymmetric.

By \eqref{eq:visual-is-chordal}, the metrics \(\varrho_o\) and
\(\varrho_{o'}\) induce the usual closed-ball boundary topologies.

It remains only to relate the induced Gromov-boundary map to the
closed-ball compactification.  Define
\begin{align}
\overline F|_{\Hn}:=F,
\qquad
\overline F|_{\bdry\Hn}:=\bdry F.
\nonumber
\end{align}
Suppose that \(z_j\in\Hn\) and \(z_j\to\xi\in\bdry\Hn\).  Then
\(\{z_j\}\) converges at infinity and represents \(\xi\).  By the
definition of the induced map in Bonk--Schramm
\cite[Proposition~6.3\textup{(2)}, pp.~282--283]{BonkSchramm},
the sequence \(\{F(z_j)\}\) represents \(\bdry F(\xi)\).  Hence
\begin{align}
F(z_j)\longrightarrow\bdry F(\xi)
\nonumber
\end{align}
in the closed-ball compactification; compare
\Cref{lem:elementary-hyperbolic-boundary-estimates}.

Thus \(\overline F\) is continuous along interior sequences approaching
the boundary.  It is continuous on \(\Hn\) by the assumed continuity of
\(F\), and it is continuous on \(\bdry\Hn\) because \(\bdry F\) is a
homeomorphism.  

Splitting a general convergent sequence into its interior
and boundary terms proves continuity on the whole compactification.
Finally, the extension is unique because \(\Hn\) is dense in
\(\Hn\cup\bdry\Hn\) and the latter space is Hausdorff.
\end{proof}

The map \(\bar u:M\to N\) is the Eells--Sampson harmonic representative fixed in
\Cref{sec:statement}, and
\[
u:\Hn\longrightarrow\Hn
\]
is its \(\rho\)-equivariant harmonic lift from
\eqref{eq:equivariance}.  By \Cref{prop:rough}, \(u\) is a globally
Lipschitz rough isometry.  

Fix \(p\in\Hn\), and put \(q:=u(p)\). By
\Cref{prop:rough-boundary-extension}, the boundary extension
\begin{align}
\bdry u:
\bigl(\bdry\Hn,\varrho_p\bigr)
\longrightarrow
\bigl(\bdry\Hn,\varrho_q\bigr)
\nonumber
\end{align}
is a quasisymmetric homeomorphism.

Let
\begin{align}
\iota_p:\bdry\Hn\longrightarrow\mathbb S^{n-1},
\qquad
\iota_q:\bdry\Hn\longrightarrow\mathbb S^{n-1}
\nonumber
\end{align}
be ball-model boundary identifications based at \(p\) and \(q\),
respectively, as in
Definition~\ref{def:visual-metric-ball-identification}, and define
\begin{align}
\phi:=\iota_q\circ(\bdry u)\circ\iota_p^{-1}:\mathbb S^{n-1}\longrightarrow\mathbb S^{n-1}.
\label{eq:spherical-boundary-map}
\end{align}

By \eqref{eq:visual-round-bilipschitz}, the source and target ball-model
boundary identifications are bi-Lipschitz between the visual and round
metrics.  Quasisymmetry is therefore preserved under these identifications.

Consequently, there is an increasing homeomorphism
\begin{align}
\eta:[0,\infty)\longrightarrow[0,\infty)
\nonumber
\end{align}
such that
\begin{align}
\frac{d_{ \mathbb{S}^{n-1}}(\phi(\xi),\phi(\alpha))}
     {d_{ \mathbb{S}^{n-1}}(\phi(\xi),\phi(\beta))}
\leq
\eta\!\left(
\frac{d_{ \mathbb{S}^{n-1}}(\xi,\alpha)}{d_{ \mathbb{S}^{n-1}}(\xi,\beta)}
\right)
\label{eq:qs_sphere}
\end{align}
whenever \(\xi,\alpha,\beta\in \mathbb{S}^{n-1}\) are distinct.

Finally, the boundary map is equivariant.  If \(r\) is a ray ending at
\(\xi\) and $\gamma\in \Gamma$, then \(\gamma r\) ends at \(\gamma\xi\), and equivariance of
\(u\) gives
\begin{align}
u(\gamma r(t))=\rho(\gamma)u(r(t)).
\nonumber
\end{align}
The target curve on the right has endpoint
\(\rho(\gamma)\phi(\xi)\).  Since the definition of \(\bdry u\) is
independent of the initial point of the representing ray, the curve on
the left has endpoint \(\phi(\gamma\xi)\).  Hence
\begin{align}
\phi(\gamma\xi)=\rho(\gamma)\phi(\xi),
\qquad
\gamma\in\Gamma,
\quad
\xi\in \mathbb{S}^{n-1}.
\nonumber
\end{align}
%\ReviewUnusedLabel{eq:boundary_equivariance}

The boundary homeomorphism \(\phi\)
defined in \eqref{eq:spherical-boundary-map} is now fixed.  The purpose of the next
section is to prove, by an elementary argument and without invoking
the standard quasiconformal theorem asserting differentiability almost
everywhere and a Jacobian that is nonzero almost everywhere, that there
exists at least one
point \(\xi\in\mathbb S^{n-1}\) at which \(D\phi_\xi\) exists and is
invertible.  The set of all such points is formally defined in
\Cref{thm:E_nonempty}.

\section{Existence of a nondegenerate tangent}

We first work in Euclidean coordinates. 

\begin{proposition}\label{prop:elementary_qs_regular}
Let \(n\geq3\), let \(\Omega,\Omega'\subset\R^{n-1}\) be domains, and let
\begin{align}
f:\Omega\longrightarrow\Omega'\nonumber
\end{align}
be an \(\eta\)-quasisymmetric homeomorphism in the sense of
Definition~\ref{def:quasisymmetry}.  Then:
\begin{enumerate}[label=\textup{(\roman*)},leftmargin=3em]
\item
for every \(1\leq p<n-1\),
\begin{align}
f\in W^{1,p}_{\mathrm{loc}}(\Omega;\R^{n-1});\nonumber
\end{align}
\item
\(f\) is differentiable at Lebesgue-almost every point of
\(\Omega\), and its weak differential agrees almost everywhere with its
classical differential;
\item
at every point \(x\) where \(f\) is   differentiable,
\begin{align}
Df_x=0
\qquad\text{or}\qquad
Df_x\text{ is invertible}.
\label{eq:zero_or_invertible}
\end{align}
\end{enumerate}
\end{proposition}

\begin{proof}
\textbf{Step (1).} 
Fix bounded open sets
\begin{align}
V\Subset W\Subset\Omega.\nonumber
\end{align}
Choose \(r_0>0\) such that
\begin{align}
B(x,2r_0)\Subset W
\qquad
\text{for every }x\in V.\nonumber
\end{align}
Define a finite Borel measure on \(W\) by
\begin{align}
\nu(A):=\mathcal L^{n-1}(f(A)).\nonumber
\end{align}
Equivalently, \(\nu\) is the pushforward of Lebesgue measure on
\(f(W)\) by the continuous inverse \(f^{-1}\); hence it is indeed a
Borel measure.  Finiteness follows from \(\overline W\Subset\Omega\),
because \(f(\overline W)\) is compact.

Let \(x\in V\), let \(y\in\Omega\), and put
\begin{align}
r:=|x-y|<r_0.\nonumber
\end{align}
For every \(z\in\partial B(x,2r)\), quasisymmetry gives
\begin{align}
\frac{|f(x)-f(y)|}{|f(x)-f(z)|}
\leq
\eta(1/2).\nonumber
\end{align}
Since \(f\) is a homeomorphism and \(\overline{B(x,2r)}\Subset\Omega\),
\begin{align}
\partial f(B(x,2r))=f(\partial B(x,2r)).\nonumber
\end{align}
It follows that
\begin{align}
B\!\left(
 f(x),\frac{|f(x)-f(y)|}{\eta(1/2)}
\right)
\subset f(B(x,2r)).\nonumber
\end{align}
Indeed, the distance from \(f(x)\) to the boundary of the open set
\(f(B(x,2r))\) is at least the displayed radius.  Taking Euclidean
volumes, with \(\omega_{n-1}=\mathcal L^{n-1}(B(0,1))\), yields
\begin{align}
\omega_{n-1}
\left(
\frac{|f(x)-f(y)|}{\eta(1/2)}
\right)^{n-1}
\leq
\nu(B(x,2r)).\nonumber
\end{align}
Consequently,
\begin{align}
|f(x)-f(y)|
\leq
C_{n-1,\eta}|x-y|
\left(
\frac{\nu(B(x,2|x-y|))}{(2|x-y|)^{n-1}}
\right)^{1/(n-1)}.
\label{eq:qs_pointwise_density}
\end{align}

Define the truncated maximal density
\begin{align}
g(x)
:=
\left(
\sup_{0<s<2r_0}
\frac{\nu(B(x,s))}{s^{n-1}}
\right)^{1/(n-1)},
\qquad x\in V.
\label{eq:g_def}
\end{align}
The function \(g\) is measurable.  By continuity from below of \(\nu\),
the supremum may be taken over rational \(s\in(0,2r_0)\); for fixed
\(s\), the function \(x\mapsto\nu(B(x,s))\) is Borel, as follows by
approximating the ball indicator with continuous radial functions and
applying monotone convergence.
Then \eqref{eq:qs_pointwise_density} gives
\begin{align}
|f(x)-f(y)|
\leq
C_{n-1,\eta}|x-y|g(x)
\label{eq:hajlasz_one_point}
\end{align}
whenever \(x\in V\) and \(|x-y|<r_0\).

\textbf{Step (2).}
For the function \(g\) defined in \eqref{eq:g_def} and
\(\lambda>0\), put
\begin{align}
A_\lambda:=\{x\in V:g(x)>\lambda\}.\nonumber
\end{align}
For each \(x\in A_\lambda\), choose \(0<r_x<2r_0\) such that
\begin{align}
\nu(B(x,r_x))>\lambda^{n-1} r_x^{n-1}.\nonumber
\end{align}
The \(5r\)-covering lemma supplies a countable pairwise disjoint
subfamily
\begin{align}
B_i=B(x_i,r_i)\nonumber
\end{align}
such that
\begin{align}
A_\lambda\subset\bigcup_i5B_i.\nonumber
\end{align}
All \(B_i\) lie in \(W\).  Therefore
\begin{align}
\mathcal L^{n-1}(A_\lambda)
&\leq
5^{n-1}\omega_{n-1}\sum_i r_i^{n-1}\leq
\frac{5^{n-1}\omega_{n-1}}{\lambda^{n-1}}
\sum_i\nu(B_i)\leq
\frac{5^{n-1}\omega_{n-1}\nu(W)}{\lambda^{n-1}}.
\label{eq:g_weak_L_boundary}
\end{align}

Combining \eqref{eq:g_weak_L_boundary} with the trivial estimate
\(\mathcal L^{n-1}(A_\lambda)\leq\mathcal L^{n-1}(V)\), we have
\begin{align}
\mathcal L^{n-1}(A_\lambda)
\leq
\min\left\{
\mathcal L^{n-1}(V),
\frac{5^{n-1}\omega_{n-1}\nu(W)}
{\lambda^{n-1}}
\right\}.
\nonumber
\end{align}
For every \(p>0\), Tonelli's theorem gives
\begin{align}
\int_Vg^p\,\dd x
=
p\int_0^\infty
\lambda^{p-1}\mathcal L^{n-1}(A_\lambda)\,\dd\lambda.
\nonumber
\end{align}

Splitting the integral at \(\lambda=1\), for \(0<p<n-1\) we obtain
\begin{align}
\begin{aligned}
\int_Vg^p\,\dd x
&\leq
p\mathcal L^{n-1}(V)
\int_0^1\lambda^{p-1}\,\dd\lambda+
p\,5^{n-1}\omega_{n-1}\nu(W)
\int_1^\infty\lambda^{p-n}\,\dd\lambda\\
&=
\mathcal L^{n-1}(V)
+
\frac{
p\,5^{n-1}\omega_{n-1}\nu(W)
}{
n-1-p
}
<\infty.
\end{aligned}
\nonumber
\end{align}
Here the integral over \((0,1)\) is finite because \(p>0\), whereas
the integral over \((1,\infty)\) is finite because \(p<n-1\).
Thus
\begin{align}
g\in L^p(V)
\qquad
\text{for every }0<p<n-1.
\label{eq:g_Lp}
\end{align}

In particular, \(g(x)<\infty\) for almost every \(x\in V\), and
\eqref{eq:hajlasz_one_point} implies
\begin{align}
\varlimsup_{y\to x}
\frac{|f(y)-f(x)|}{|y-x|}
\leq C_{n-1,\eta}g(x)<\infty
\label{eq:finite_upper_lip}
\end{align}
for almost every \(x\in V\).

\textbf{Step (3).} We now prove directly that \eqref{eq:finite_upper_lip} implies  
differentiability almost everywhere.  For integers \(a,b\geq1\), let
\(E_{a,b}\subset V\) consist of all \(x\in V\) such that
\begin{align}
|f(x)|\leq a\nonumber
\end{align}
and
\begin{align}
|f(y)-f(x)|\leq a|y-x|
\label{eq:Eab_local_lip}
\end{align}
whenever \(y\in\Omega\) and \(0<|y-x|<1/b\).  These sets are measurable;
by continuity, the condition may be tested on a fixed countable dense
subset of \(\Omega\).  The set on which the upper pointwise Lipschitz
constant is finite is contained in
\begin{align}
\bigcup_{a,b\geq1}E_{a,b}.\nonumber
\end{align}

The restriction \(f|_{E_{a,b}}\) is globally Lipschitz.  Indeed, if
\(x,y\in E_{a,b}\) and \(|x-y|<1/b\), then
\eqref{eq:Eab_local_lip} applies; if \(|x-y|\geq1/b\), then
\begin{align}
|f(x)-f(y)|\leq2a\leq2ab|x-y|.\nonumber
\end{align}

For each nonempty
\(E_{a,b}\), let \(L_{a,b}:=\max\{a,2ab\}\).  For each coordinate function of
\(f|_{E_{a,b}}\), the McShane formula
\begin{align}
F^\alpha(z)
:=
\inf_{x\in E_{a,b}}
\bigl(f^\alpha(x)+L_{a,b}|z-x|\bigr)\nonumber
\end{align}
defines an \(L_{a,b}\)-Lipschitz extension to \(\R^{n-1}\).  Indeed, the
Lipschitz inequality for \(f^\alpha|_{E_{a,b}}\) shows that
\(F^\alpha=f^\alpha\) on \(E_{a,b}\), while the triangle inequality
shows that
\begin{align}
|F^\alpha(z)-F^\alpha(z')|\leq L_{a,b}|z-z'|.\nonumber
\end{align}
The vector-valued extension
\begin{align}
F=(F^1,\dots,F^{n-1})\nonumber
\end{align}
is \(\sqrt{n-1}\,L_{a,b}\)-Lipschitz, since every coordinate is
\(L_{a,b}\)-Lipschitz.  Rademacher's theorem therefore makes \(F\)
  differentiable almost everywhere.

Let \(x\in E_{a,b}\) be both a density point of \(E_{a,b}\) and a
  differentiability point of \(F\).  We claim that \(f\) itself is
differentiable at \(x\), with derivative \(DF_x\).  Given any sequence
\(y_k\to x\), the density of \(E_{a,b}\) at \(x\) allows us to choose
\(z_k\in E_{a,b}\) such that
\begin{align}
|z_k-y_k|=o(|y_k-x|).
\label{eq:zk_approx_yk}
\end{align}
Indeed, otherwise some \(\varepsilon>0\) and a subsequence would
satisfy
\begin{align}
B(y_k,\varepsilon|y_k-x|)\cap E_{a,b}=\varnothing,\nonumber
\end{align}
which would contradict that \(x\) is a density point.  

For large \(k\),
\eqref{eq:Eab_local_lip} at \(z_k\) and
\eqref{eq:zk_approx_yk} give
\begin{align}
|f(y_k)-f(z_k)|
\leq
a|y_k-z_k|
=
o(|y_k-x|).\nonumber
\end{align}
Since \(F=f\) on \(E_{a,b}\) and \(F\) is differentiable at \(x\),
\begin{align}
f(z_k)-f(x)
=
DF_x(z_k-x)+o(|z_k-x|).\nonumber
\end{align}
Together with \eqref{eq:zk_approx_yk}, this yields
\begin{align}
f(y_k)-f(x)
=
DF_x(y_k-x)+o(|y_k-x|).\nonumber
\end{align}
Thus \(f\) is   differentiable at \(x\).  The Lebesgue density
theorem, Rademacher's theorem, and the countable union over \(a,b\) show
that \(f\) is   differentiable almost everywhere in \(V\).
Since \(V\Subset\Omega\) was arbitrary, the same holds in \(\Omega\).

\textbf{Step (4).} Fix \(1\leq p<n-1\) and \(K\Subset V\).  For a coordinate vector \(e_i\)
and sufficiently small nonzero \(h\), define
\begin{align}
D_{i,h}f(x)
:=
\frac{f(x+he_i)-f(x)}{h},
\qquad x\in K.\nonumber
\end{align}
By \eqref{eq:hajlasz_one_point},
\begin{align}
|D_{i,h}f(x)|\leq C_{n-1,\eta}g(x)\nonumber
\end{align}
for almost every \(x\in K\).  At almost every differentiability point,
\begin{align}
D_{i,h}f(x)\longrightarrow Df_x(e_i)
\qquad(h\to0),\nonumber
\end{align}
and the same bound holds for \(|Df_x(e_i)|\).  Since \(g\in L^p(K)\) by \eqref{eq:g_Lp},
dominated convergence gives
\begin{align}
D_{i,h}f
\longrightarrow
Df(e_i)
\qquad
\text{in }L^p(K;\R^{n-1}).
\label{eq:difference_quotient_Lp}
\end{align}
For \(\psi\in C_c^\infty(K)\) and sufficiently small \(h\), a change of
variables gives
\begin{align}
\int_KD_{i,h}f(x)\psi(x)\,\dd x
=
\int_\Omega
f(x)
\frac{\psi(x-he_i)-\psi(x)}{h}
\,\dd x.\nonumber
\end{align}
Letting \(h\to0\) and using \eqref{eq:difference_quotient_Lp}, we obtain
\begin{align}
\int_KDf_x(e_i)\psi(x)\,\dd x
=
-\int_Kf(x)\partial_i\psi(x)\,\dd x.\nonumber
\end{align}
Hence \(Df(e_i)\) is the \(i\)-th weak derivative.  This proves
\begin{align}
f\in W^{1,p}_{\mathrm{loc}}(\Omega;\R^{n-1})\nonumber
\end{align}
and shows that the weak and classical differentials agree almost
everywhere.

\textbf{Step (5).} Zero or invertible.
Let \(x\in\Omega\) be a differentiability point, and let
\(v,w\in\mathbb S^{n-2}\).  For sufficiently small \(t\neq0\), the points
\(x+tv\) and \(x+tw\) are equally distant from \(x\).  Quasisymmetry gives
\begin{align}
|f(x+tv)-f(x)|
\leq
\eta(1)|f(x+tw)-f(x)|.\nonumber
\end{align}
Divide by \(|t|\) and let \(t\to0\):
\begin{align}
|Df_x(v)|
\leq
\eta(1)|Df_x(w)|.
\label{eq:linear_distortion_at_derivative}
\end{align}
If \(Df_x(w)=0\) for one unit vector \(w\), then
\eqref{eq:linear_distortion_at_derivative} implies \(Df_x(v)=0\) for every
unit vector \(v\), so \(Df_x=0\).  Otherwise \(Df_x\) has trivial kernel;
because its domain and range both have dimension \(n-1\), it is invertible.
This proves \eqref{eq:zero_or_invertible}.
\end{proof}

\begin{theorem}\label{thm:E_nonempty}
Let \(n\geq 3\), define
\begin{align}
E_\phi:=\left\{\xi\in \mathbb{S}^{n-1}: D\phi_\xi\text{ exists and }D\phi_\xi\text{ is invertible}\right\}.
\nonumber
\end{align}
Then \(E_\phi\neq\varnothing\).
\end{theorem}

\begin{proof}
\textbf{Step (1)}. By \eqref{eq:qs_sphere}, \(\phi\) is quasisymmetric for the
round geodesic distance.

For each \(\xi\in \mathbb{S}^{n-1}\), choose smooth source and target charts
\begin{align}
\alpha:U\longrightarrow\Omega\subset\R^{n-1},
\qquad
\beta:V\longrightarrow\widetilde\Omega'\subset\R^{n-1},
\nonumber
\end{align}
and precompact neighborhoods
\(\xi\in U_0\Subset U_1\Subset U\) such that
\(\phi(\overline{U_1})\subset V\).  On the compact sets
\(\overline{U_1}\) and \(\phi(\overline{U_1})\), the charts and their
inverses are bi-Lipschitz.  Hence the restricted coordinate map
\begin{align}
\beta\circ\phi\circ\alpha^{-1}:
\alpha(U_1)\longrightarrow\beta(\phi(U_1))
\nonumber
\end{align}
is a quasisymmetric homeomorphism between Euclidean domains, with a control
function depending only on the original control function and the relevant
bi-Lipschitz constants.

Applying Proposition~\ref{prop:elementary_qs_regular} to this restricted map
gives all conclusions on \(U_0\).  Compactness of \( \mathbb{S}^{n-1}\) supplies finitely
many such sets \(U_0\) covering \( \mathbb{S}^{n-1}\).

Smooth changes of coordinates preserve   differentiability and the rank of the differential, while Sobolev regularity is invariant under such coordinate changes. Hence all the asserted conclusions hold
globally on \( \mathbb{S}^{n-1}\).

Then, in smooth local charts,
\begin{align}
\phi\in W^{1,p}_{\mathrm{loc}}
\qquad
\text{for every }1\leq p<n-1.\label{regularity-of-phi}
\end{align}
The map \(\phi\) is   differentiable at round-almost every point,
and at every differentiability point
\begin{align}
D\phi_\xi=0
\qquad\text{or}\qquad
D\phi_\xi\text{ is invertible}.\nonumber
\end{align}
The weak and classical differentials agree almost everywhere in every
chart.

\textbf{Step (2)}.  Now assume for contradiction that
\begin{align}
E_\phi=\varnothing.\nonumber
\end{align}

Then $D\phi_\xi=0$ for round-almost every $\xi\in \mathbb{S}^{n-1}$.

Fix \(\xi_0\in \mathbb{S}^{n-1}\).  Choose a connected source neighborhood
\(U\ni\xi_0\) and smooth coordinate charts
\begin{align}
\alpha:U\longrightarrow\Omega\subset\R^{n-1},
\qquad
\beta:V\longrightarrow\widetilde\Omega'\subset\R^{n-1}\nonumber
\end{align}
such that \(\phi(U)\Subset V\).  Put
\begin{align}
f:=\beta\circ\phi\circ\alpha^{-1}:
\Omega\longrightarrow\Omega':=\beta(\phi(U)).\nonumber
\end{align}
Smooth coordinate changes preserve null sets locally and preserve the
rank of a differential.  Hence 
\begin{align}
Df=0
\qquad\text{for Lebesgue-almost every point of }\Omega.
\label{eq:Df_zero_ae}
\end{align}

Since \(n\geq3\), we have \(1<n-1\).  Applying
\eqref{regularity-of-phi} with \(p=1\), we obtain
\begin{align}
f\in W^{1,1}_{\mathrm{loc}}(\Omega;\R^{n-1}).
\nonumber
\end{align}
Moreover, its weak differential agrees almost everywhere with its
classical differential.

%Thus the weak differential of \(f\) vanishes almost
%everywhere by \eqref{eq:Df_zero_ae}.

By \eqref{eq:Df_zero_ae}, the weak
differential \(D^{w}f\) vanishes almost everywhere in \(\Omega\).
Equivalently, for every coordinate function \(f^\alpha\) of
\(f\) and every \(i\in\{1,\ldots,n-1\}\),
\begin{align}
\partial_i f^\alpha=0
\qquad\text{in }\mathcal D'(\Omega).
\nonumber
\end{align}

By the classical constancy theorem for distributions
\cite[Theorem~6.11]{LiebLoss}, applied to the coordinate functions of \(f\),
the map \(f\) is almost everywhere equal to a constant vector on each
connected component of \(\Omega\). 

Since \(f\) is continuous, this equality holds everywhere on each component. Consequently, \(f\) is
locally constant.  This contradicts that \(\phi\) is a homeomorphism.
The contradiction proves \(E_\phi\neq\varnothing\).
\end{proof}

\begin{remark}[The precise role of \(n\geq3\)]
In the present proof, the assumption \(n\geq3\) is used only in the
passage from quasisymmetry of the boundary map to the existence of a
nondegenerate boundary tangent.

The preceding boundary-dimensional obstruction has a classical
antecedent, but the present proof uses it more economically.  Mostow
explicitly observed that \(n>2\) entered his original proof twice:
first, in the theorem that the Jacobian of a quasiconformal boundary
map is positive almost everywhere, and second, in the
conformal-capacity argument that forces its tangent map to be
conformal; see
\cite[Theorem~9.4 and \S12, pp.~94, 99--102]{Mostow1968}.
The first input is replaced here by
Proposition~\ref{prop:elementary_qs_regular} and
Theorem~\ref{thm:E_nonempty}, which produce the single nondegenerate
tangent needed below.  The second input is unnecessary in the present
argument: the tangent \(A\) may be anisotropic, while its exact
harmonic extension \(H_A\) nevertheless satisfies
\begin{align}
|dH_A|^2=n.
\nonumber
\end{align}

By contrast, in the real-hyperbolic Besson--Courtois--Gallot
natural-map argument, the restriction \(n\geq3\) is concentrated in
the determinant inequality
\begin{align}
\frac{\det H}{\det(I-H)^2}
\leq
\left(\frac{n}{(n-1)^2}\right)^n,
\qquad
\operatorname{tr}H=1,
\nonumber
\end{align}
whose equality case is \(H=\frac1n I\); compare
\cite[pp.~637--639]{BessonCourtoisGallot}.  This inequality is false
when \(n=2\).  Thus the surface exception appears in different
analytic modules in the two proofs: boundary regularity here and the
natural-map Jacobian estimate in the barycenter method.
\end{remark}

% ================================================================
\section{Boundary-tangent renormalization of the harmonic map}
\label{sec:tangent}
% ================================================================

The purpose of this section is to recenter the harmonic map along a
geodesic escaping to one ideal endpoint and to identify the resulting
boundary-tangent limit.  

We first record three elementary facts about real hyperbolic space.  They
are included to make the two compactness passages below explicit.

\begin{lemma}\label{lem:visual_facts}
Let \(d\geq2\).
\begin{enumerate}[label=\textup{(\alph*)},leftmargin=3em]
\item\label{item:bounded_same_endpoint}
If \(x_j,y_j\in\mathbb H^d\),
\(\sup_jd(x_j,y_j)<\infty\), and
\(x_j\to\zeta\in\partial_\infty\mathbb H^d\), then
\(y_j\to\zeta\).

\item\label{item:moving_geodesics}
Suppose \(a_j,b_j\in\partial_\infty\mathbb H^d\),
\(a_j\to a\), \(b_j\to b\), and \(a\neq b\).  If
\(z_j\in(a_jb_j)\) lies on the complete geodesic with endpoints
\(a_j,b_j\), and \(z_j\) leaves every compact subset, then every boundary
accumulation point of \(z_j\) belongs to \(\{a,b\}\).

\item\label{item:moving_rays}
If \(p_j\to p\in\mathbb H^d\) and
\(\zeta_j\to\zeta\in\partial_\infty\mathbb H^d\), then the unit-speed
geodesic rays from \(p_j\) to \(\zeta_j\) converge, after using the common
parameter origin \(0\), in \(C^\infty\) on every compact parameter
interval to the ray from \(p\) to \(\zeta\).
\end{enumerate}
\end{lemma}

\begin{proof}
Use the Poincar\'e ball model and its closed-ball compactification.  For
\ref{item:bounded_same_endpoint}, the distance formula
\[
\sinh\frac{d(x,y)}2
=
\frac{|x-y|}{\sqrt{(1-|x|^2)(1-|y|^2)}}
\]
shows that a hyperbolically bounded ball about a point approaching the
Euclidean unit sphere has Euclidean diameter tending to zero.  Thus
bounded-distance sequences have the same ideal limit.

For \ref{item:moving_geodesics}, the closure of a hyperbolic geodesic is
the Euclidean circular arc or diameter orthogonal to the unit sphere.  As
long as the two endpoints remain distinct, that closed arc depends
continuously, in the Hausdorff topology, on its ordered endpoint pair.  Hence
the only points of the limiting closed arc lying on the unit sphere are
\(a\) and \(b\).

For \ref{item:moving_rays}, the initial unit vector at \(p\) pointing to
an ideal endpoint \(\zeta\) depends smoothly on \(p\) and continuously on
\(\zeta\) in the ball model.  Continuous dependence for the geodesic ODE,
followed by differentiation of the ODE, gives \(C^\infty\)-convergence on
compact parameter intervals.
\end{proof}

By \Cref{thm:E_nonempty}, choose
\begin{align}
\xi\in E_\phi.
\nonumber
\end{align}

By the definition of \(E_\phi\), the spherical boundary map
\(\phi:\mathbb S^{n-1}\to\mathbb S^{n-1}\) is differentiable at
\(\xi\) in smooth boundary charts, and
\begin{align}
D\phi_\xi:
T_\xi\mathbb S^{n-1}
\longrightarrow
T_{\phi(\xi)}\mathbb S^{n-1}
\nonumber
\end{align}
is invertible.

Recall that
\begin{align}
\phi
=
\iota_q\circ(\bdry u)\circ\iota_p^{-1},
\nonumber
\end{align}
where
\begin{align}
\iota_p,\iota_q:
\bdry\Hn\longrightarrow\mathbb S^{n-1}
\nonumber
\end{align}
are the ball-model boundary identifications introduced above.  Let
\begin{align}
\widehat\xi
:=
\iota_p^{-1}(\xi)
\in\bdry\Hn
\nonumber
\end{align}
be the corresponding ideal boundary point of the source copy of
\(\Hn\).  Choose a unit-speed geodesic ray
\begin{align}
r:[0,\infty)\longrightarrow\Hn,
\qquad
r(\infty)=\widehat\xi.
\nonumber
\end{align}

We first make a fixed boundary-coordinate normalization of the map.
Identify both the source and the target with
the upper half-space model
\begin{align}
\Hn
&=
\left\{
(x,y):x\in\R^{n-1},\ y>0
\right\},
&
g_{\mathbb H}
&=
\frac{|dx|^2+dy^2}{y^2}.
\nonumber
\end{align}
In this model, the ideal boundary has the standard smooth conformal
identification
\begin{align}
\bdry\Hn
=
\R^{n-1}\cup\{\infty\}
\cong
\mathbb S^{n-1},
\nonumber
\end{align}
where the last identification may be realized by inverse stereographic
projection.  

We denote the origin of \(\R^{n-1}\) by
\(\mathbf{0}\). In particular, \(\mathbf{0}\in\R^{n-1}\) and \(\infty\) are ideal
boundary points, and the finite part \(\R^{n-1}\) is a smooth boundary
chart.

Let \(\widehat\xi^{-}\) be the negative endpoint of the complete
geodesic containing \(r\).  Choose isometries
\begin{align}
G,H\in\operatorname{Isom}(\Hn)
\nonumber
\end{align}
such that
\begin{align}
(\bdry G)(\widehat\xi)
&= \mathbf{0},
&
(\bdry G)(\widehat\xi^{-})
&=\infty,
\nonumber\\
(\bdry H)\bigl((\bdry u)(\widehat\xi)\bigr)
&=\mathbf{0},
&
(\bdry H)\bigl((\bdry u)(\widehat\xi^{-})\bigr)
&=\infty.
\nonumber
\end{align}
Define
\begin{align}
\widetilde u
&:=
H\circ u\circ G^{-1},
&
\widetilde r
&:=
G\circ r.
\nonumber
\end{align}
The boundary map of \(\widetilde u\) is
\begin{align}
\widetilde\phi
:=
\bdry\widetilde u
=
(\bdry H)\circ(\bdry u)\circ(\bdry G)^{-1}.
\nonumber
\end{align}

To compare this normalized boundary map with the spherical map
\(\phi\), set
\begin{align}
\alpha
&:=
\iota_p\circ(\bdry G)^{-1}:
\R^{n-1}\cup\{\infty\}
\longrightarrow
\mathbb S^{n-1},
\nonumber\\
\beta
&:=
(\bdry H)\circ\iota_q^{-1}:
\mathbb S^{n-1}
\longrightarrow
\R^{n-1}\cup\{\infty\}.
\nonumber
\end{align}
Then
\begin{align}
\widetilde\phi
=
\beta\circ\phi\circ\alpha,
\qquad
\alpha(\mathbf{0})=\xi,
\qquad
\beta(\phi(\xi))=\mathbf{0}.
\nonumber
\end{align}
The maps \(\alpha\) and \(\beta\) are smooth conformal
diffeomorphisms of the corresponding ideal boundaries.  Therefore the
chain rule gives
\begin{align}
D\widetilde\phi_{\mathbf{0}}
=
D\beta_{\phi(\xi)}
\circ
D\phi_\xi
\circ
D\alpha_{\mathbf{0}}.
\nonumber
\end{align}
Since all three factors on the right-hand side are invertible,
\(D\widetilde\phi_{\mathbf{0}}\) is invertible.

\begin{notation}
%\ReviewUnusedLabel{notation:phi-tilde-phi}
{To avoid introducing permanent new notation, from now on we write
\(u\), \(r\), and \(\phi\) for
\(\widetilde u\), \(\widetilde r\), and \(\widetilde\phi\),
respectively, and we write \(\xi^{-}\) for the negative endpoint of the
complete geodesic containing the normalized ray.  
}
\end{notation}

Thus
\begin{equation}\label{eq:normalized_endpoints}
r(\infty)=\mathbf{0},
\qquad
\xi^{-}=\infty,
\qquad
\phi(\mathbf{0})
=\mathbf{0},
\qquad
\phi(\infty)=\infty.
\end{equation}

After translating the ray parameter, we may assume
\begin{equation}\label{eq:vertical_ray}
r(t)=(\mathbf{0},e^{-t}).
\end{equation}
Hyperbolic isometries preserve harmonicity, the rough-isometry constants,
and the Hilbert--Schmidt norm of the differential.

By \eqref{eq:normalized_endpoints}, \(\phi(\infty)=\infty\).  Since
\(\phi\) is a boundary homeomorphism, its restriction to the finite
boundary chart is a map
\begin{align}
\phi|_{\R^{n-1}}:
\R^{n-1}\longrightarrow\R^{n-1}.
\nonumber
\end{align}

We define
\begin{align}
A
:=
D\left(\phi|_{\R^{n-1}}\right)_{\mathbf{0}}
=
D\phi_{\mathbf{0}}
\in\operatorname{GL}(n-1,\R),
\nonumber
\end{align}
where the second expression uses the identity chart on the finite boundary.

We next perform a scale-dependent boundary-tangent renormalization.
For \(s>0\), let
\begin{align}
\delta_s(x,y):=(sx,sy):\Hn\longrightarrow\Hn.
\nonumber
\end{align}
This is a hyperbolic isometry.

Define
\begin{equation}\label{eq:Us_def}
u_s:=\delta_{1/s}\circ u\circ\delta_s.
\end{equation}
Its boundary map is
\begin{equation}\label{eq:fs_def}
\phi_s(z)=\frac{\phi(sz)}s,
\qquad
\phi_s(\infty)=\infty.
\end{equation}

Differentiability at \(\mathbf{0}\) says
\[
\phi(z)=Az+\varepsilon(z)|z|,
\qquad
\varepsilon(z)\longrightarrow
\mathbf{0}
\quad
\bigl(z\to\mathbf{0}\bigr).
\]

By \eqref{eq:fs_def}, if \(K\Subset\R^{n-1}\), then
\[
\lim_{s\rightarrow 0}\sup_{x\in K}|\phi_s(x)-Ax| \leq \lim_{s\rightarrow 0}\left(\sup_{|z|\leq s\sup_K|x|}|\varepsilon(z)|\right)\sup_K|x| =0.
\]
Therefore
\begin{equation}\label{eq:fs_to_A}
\phi_s\longrightarrow A
\qquad\text{locally uniformly on }\R^{n-1},
\end{equation}
while \(\phi_s(\infty)=A(\infty)=\infty\).  

All \(u_s\) are harmonic,
share the global Lipschitz constant of \(u\), and share its two-sided
rough-isometry and coarse-surjectivity constants.

\begin{lemma}\label{lem:basepoint_compact}
Let $o:=(\mathbf{0},1)\in\Hn$.  Then the set
$\{u_s(o):0<s\leq1\}$ is contained in a compact subset of \(\Hn\).
\end{lemma}

\begin{proof}
\textbf{Step (1)}. Fix a unit vector \(e\in\R^{n-1}\).  Since \(A\) is invertible,
\(Ae\neq\mathbf{0}\).  Let \(L_0\) be the complete geodesic with endpoints
\(\mathbf{0},\infty\), and let \(L_1\) be the complete geodesic with endpoints
\(-e,e\).  Both pass through \(o\).

For \(i\in\{0,1\}\), choose a unit-speed parameterization
\begin{align}
\ell_i:\mathbb R\longrightarrow L_i,
\qquad
\ell_i(0)=o.
\nonumber
\end{align}
Since all the maps \(u_s\) have the same rough-isometry constants
\(L\geq1\) and \(C\geq0\), we have
\begin{align}
L^{-1}|t-t'|-C
&\leq
d\bigl(u_s(\ell_i(t)),u_s(\ell_i(t'))\bigr)
\leq
L|t-t'|+C.
\label{eq:Us-restricted-quasi-geodesic}
\end{align}
The estimate \eqref{eq:Us-restricted-quasi-geodesic} holds for every
\(t,t'\in\mathbb R\), every \(i\in\{0,1\}\), and every
\(0<s\leq1\).  Hence \(u_s\circ\ell_i\) is an
\((L,C)\)-quasi-geodesic line in the sense of
Definition~\ref{def:quasi-geodesic}.  Its two ideal endpoints are the
images under \(\phi_s=\partial_\infty u_s\) of the two endpoints of \(L_i\).

By the complete-line statement
\eqref{eq:Morse-line-conclusion} in \Cref{lem:Morse}, there is
\begin{align}
R_0=R_0(L,C)<\infty
\nonumber
\end{align}
such that, independently of \(s\),
\begin{align}
d_{\mathrm H}\bigl(u_s(L_i),G_{i,s}\bigr)\leq R_0,
\qquad i\in\{0,1\},
\nonumber
\end{align}
where \(G_{i,s}\) is the complete geodesic whose endpoints are the
\(\phi_s\)-images of the endpoints of \(L_i\).  Since
\(\ell_i(0)=o\), it follows that
\begin{align}
\dist\bigl(u_s(o),G_{0,s}\bigr)&\leq R_0,
&
\dist\bigl(u_s(o),G_{1,s}\bigr)&\leq R_0,
\label{eq:close_two_geodesics}
\end{align}
where
$
G_{0,s}=L_0
$
and \(G_{1,s}\) has endpoints \(\phi_s(-e),\phi_s(e)\).

\textbf{Step (2)}. Suppose that \(\{u_{s_j}(o)\}_{j\in\mathbb N}\) leaves every compact subset.  

If a subsequence
of \(s_j\) converged to \(s_\infty>0\), continuity of
\(s\mapsto u_s(o)\) would contradict divergence.  

Hence, after passing to
a subsequence, \(s_j\to0\).  Compactness of the closed-ball
compactification gives
\[
u_{s_j}(o)\longrightarrow\zeta
\in\partial_\infty\Hn.
\]
By \eqref{eq:close_two_geodesics}, choose
\(z_j\in G_{0,s_j}=L_0\) and \(w_j\in G_{1,s_j}\) with
\[
d(u_{s_j}(o),z_j)\leq R_0,
\qquad
d(u_{s_j}(o),w_j)\leq R_0.
\]
By \Cref{lem:visual_facts}(\ref{item:bounded_same_endpoint}), both
\(z_j\) and \(w_j\) converge to \(\zeta\).  Since a divergent sequence
on \(L_0\) can approach only an endpoint,
\[
\zeta\in\{\mathbf{0},\infty\}.
\]
On the other hand, \eqref{eq:fs_to_A} gives
\[
\phi_{s_j}(-e)\to-Ae,
\qquad
\phi_{s_j}(e)\to Ae.
\]
The two limiting endpoints are distinct.  Applying
\Cref{lem:visual_facts}(\ref{item:moving_geodesics}) to \(w_j\) yields
\[
\zeta\in\{-Ae,Ae\}.
\]
This is impossible because
\(Ae\neq\mathbf{0}\) and all four points are viewed in
\(\R^{n-1}\cup\{\infty\}\).  Therefore the normalized basepoints remain in
a compact set.
\end{proof}

\begin{lemma}\label{lem:Us_compactness}
For every sequence \(s_j\downarrow0\), there is a subsequence, still
denoted \(s_j\), and a globally Lipschitz harmonic rough isometry
$
V:\Hn\longrightarrow\Hn$ such that
\[
u_{s_j}\longrightarrow V
\qquad\text{in }C^\infty_{\mathrm{loc}}(\mathbb{H}^n).
\]
Moreover, $V$ has the same two-sided coarse constants and the same
coarse-surjectivity constant as the family \(u_s\).
\end{lemma}

\begin{proof}
\textbf{Step (1)}. By \Cref{lem:basepoint_compact}, after passing
to a subsequence we may assume
\begin{align}
u_{s_j}(o)\longrightarrow q_\infty\in\Hn.
\nonumber
\end{align}
Write
\begin{align}
q_\infty=(a_\infty,b_\infty)\in\mathbb R^{n-1}\times(0,\infty).
\nonumber
\end{align}

Fix \(R>0\), and put
\begin{align}
D_R&:=R+2,
&
Q_R&:=L_{\mathrm{Lip}}D_R+1,
\nonumber
\end{align}
where \(L_{\mathrm{Lip}}\) is the common global Lipschitz constant of
the maps \(u_s\).  Since \(u_{s_j}(o)\to q_\infty\), after discarding finitely many
terms and relabeling we may assume
\begin{align}
d\bigl(u_{s_j}(o),q_\infty\bigr)\leq1
\qquad\text{for every }j.
\nonumber
\end{align}

Consequently, for every \(z\in B(o,D_R)\),
\begin{align}
d\bigl(u_{s_j}(z),q_\infty\bigr)
&\leq
d\bigl(u_{s_j}(z),u_{s_j}(o)\bigr)
+d\bigl(u_{s_j}(o),q_\infty\bigr)\leq
L_{\mathrm{Lip}}d(z,o)+1
\leq Q_R.
\nonumber
\end{align}
Thus
\begin{align}
u_{s_j}\bigl(B(o,D_R)\bigr)
\subset B(q_\infty,Q_R)
\label{eq:Us-fixed-target-ball}
\end{align}
for every \(j\).

We now convert the intrinsic Lipschitz estimate into a Euclidean
coordinate estimate.  In the upper half-space model, the distance
formula is
\begin{align}
\cosh d_{\mathbb H}\bigl((a,b),(X,Y)\bigr)
=
1+\frac{|X-a|^2+(Y-b)^2}{2bY}.
\nonumber
\end{align}
It follows that the hyperbolic ball \(B((a,b),T)\) is the Euclidean
ball
\begin{align}
B((a,b),T)
=
\left\{(X,Y):
|X-a|^2+(Y-b\cosh T)^2<b^2\sinh^2T
\right\}.
\label{eq:hyperbolic-ball-Euclidean-description}
\end{align}
In particular, if \(z=(x,y)\in B(o,D_R)\), then
\begin{align}
e^{-D_R}\leq y\leq e^{D_R},
\qquad
|x|\leq\sinh D_R.
\label{eq:source-coordinate-bounds}
\end{align}
Write
\begin{align}
u_{s_j}(x,y)
=
\bigl(X_j(x,y),Y_j(x,y)\bigr).
\nonumber
\end{align}
By \eqref{eq:Us-fixed-target-ball} and
\eqref{eq:hyperbolic-ball-Euclidean-description},
\begin{align}
b_\infty e^{-Q_R}
&\leq Y_j(x,y)\leq b_\infty e^{Q_R},
&
|X_j(x,y)-a_\infty|
&\leq b_\infty\sinh Q_R.
\label{eq:target-coordinate-bounds}
\end{align}
Hence the coordinate functions themselves satisfy
\begin{align}
\sup_{z\in B(o,D_R)}
|u_{s_j}(z)|_{\mathbb R^{n}}
\leq C_{\mathrm{val}}(R):=
\left[
\bigl(|a_\infty|+b_\infty\sinh Q_R\bigr)^2
+b_\infty^2e^{2Q_R}
\right]^{1/2}.
\label{eq:Us-coordinate-C0-bound}
\end{align}

Because \(\delta_s\) is a hyperbolic isometry and
\begin{align}
u_s=\delta_{1/s}\circ u\circ\delta_s,
\nonumber
\end{align}
all the maps \(u_s\) have the same intrinsic Lipschitz constant:
\begin{align}
d\bigl(u_s(z),u_s(z')\bigr)
\leq L_{\mathrm{Lip}}d(z,z')
\qquad(z,z'\in\Hn).
\nonumber
\end{align}
Since \(u_s\) is smooth, differentiating this inequality along
geodesics gives
\begin{align}
\lVert du_s(z)\rVert_{\operatorname{op},g_{\mathbb H},g_{\mathbb H}}
\leq L_{\mathrm{Lip}}
\qquad(z\in\Hn).
\label{eq:Us-intrinsic-differential-bound}
\end{align}

Let \(v\in T_{(x,y)}\Hn\cong\mathbb R^{n}\).  From
\begin{align}
g_{\mathbb H,(x,y)}(v,v)
&=\frac{|v|_{\mathbb R^{n}}^2}{y^2},
&
g_{\mathbb H,u_{s_j}(x,y)}
\bigl(du_{s_j}(v),du_{s_j}(v)\bigr)
&=
\frac{|D_{\mathrm E}u_{s_j}(v)|_{\mathbb R^{n}}^2}
{Y_j(x,y)^2},
\nonumber
\end{align}
where \(D_{\mathrm E}\) denotes the ordinary Euclidean coordinate
differential, estimate \eqref{eq:Us-intrinsic-differential-bound} gives
\begin{align}
\frac{|D_{\mathrm E}u_{s_j}(v)|_{\mathbb R^{n}}}{Y_j(x,y)}
\leq
L_{\mathrm{Lip}}
\frac{|v|_{\mathbb R^{n}}}{y}.
\nonumber
\end{align}
Combining this with \eqref{eq:source-coordinate-bounds} and
\eqref{eq:target-coordinate-bounds}, we obtain
\begin{align}
\lVert D_{\mathrm E}u_{s_j}(x,y)\rVert_{\operatorname{op}}
&\leq
L_{\mathrm{Lip}}\frac{Y_j(x,y)}{y}\leq
L_{\mathrm{Lip}}b_\infty e^{Q_R+D_R}
=:C_{\mathrm{der}}(R)
\label{eq:Us-Euclidean-first-derivative-bound}
\end{align}
for every \((x,y)\in B(o,D_R)\) and every \(j\).
Combining \eqref{eq:Us-coordinate-C0-bound} and
\eqref{eq:Us-Euclidean-first-derivative-bound}, we obtain
\begin{align}
\sup_j\left(
\|u_{s_j}\|_{L^\infty(B(o,R+2);\mathbb R^{n})}
+
\|D_{\mathrm E}u_{s_j}\|_{L^\infty(B(o,R+2))}
\right)
\leq
C_{\mathrm{val}}(R)+C_{\mathrm{der}}(R)
<\infty.
\label{eq:Us-uniform-coordinate-W1infty}
\end{align}

This is the required uniform Euclidean
\(W^{1,\infty}\)-bound for the coordinate functions of \(u_{s_j}\).
Moreover, \eqref{eq:source-coordinate-bounds} and
\eqref{eq:target-coordinate-bounds} show quantitatively that the source
and target vertical coordinates remain uniformly bounded above and
uniformly bounded away from zero.  Hence the coefficients of the source
and target hyperbolic metrics, their inverses, and all their coordinate
derivatives are uniformly bounded on the regions under consideration.

\textbf{Step (2)}. We now make the elliptic compactness argument
explicit.  Fix \(R>0\), and set
\[
K_R:=\overline{B(q_\infty,Q_R)}.
\]
Then
\[
u_{s_j}\bigl(B(o,R+2)\bigr)\subset K_R
\qquad\text{for every }j.
\]
Since upper half-space coordinates form a global smooth chart on both the
source and the target, there are precompact coordinate domains
\(\mathcal U_R,\mathcal V_R\Subset\Hn\) such that
\[
\overline{B(o,R+2)}\subset\mathcal U_R,
\qquad
K_R\subset\mathcal V_R.
\]
Thus the finite-chart localization required below consists, in the
present hyperbolic setting, of one precompact source coordinate domain and
one precompact target coordinate domain.

Choose finitely many Euclidean balls, in the source upper half-space
coordinates,
\[
\Omega_\ell^{(0)}\Subset\Omega_\ell^{(1)}
\Subset\Omega_\ell^{(2)}\Subset B(o,R+2),
\qquad 1\leq \ell\leq N_R,
\]
such that
\[
\overline{B(o,R+1)}
\subset \bigcup_{\ell=1}^{N_R}\Omega_\ell^{(0)}.
\]
On \(\Omega_\ell^{(2)}\), write
\(u_j^\alpha:=u_{s_j}^\alpha\), and let \(D\) denote the Euclidean
coordinate derivative.  In source and target coordinates the harmonic-map
equation is
\begin{equation}\label{eq:harmonic_system_local}
 g^{ab}(x)\,\partial_{ab}u_j^\alpha
 =F^\alpha\bigl(x,u_j,Du_j\bigr),
 \qquad 1\leq\alpha\leq n,
\end{equation}
where repeated indices are summed and
\begin{align}
F^\alpha(x,z,P)
&:=
 g^{ab}(x)\Gamma^c_{ab}(x)P_c^\alpha
 -g^{ab}(x)\overline\Gamma^\alpha_{\beta\gamma}(z)
   P_a^\beta P_b^\gamma.
\label{eq:harmonic-system-nonlinearity}
\end{align}
Here \(\Gamma^c_{ab}\) and
\(\overline\Gamma^\alpha_{\beta\gamma}\) are the Christoffel symbols
of the source and target hyperbolic metrics, respectively.  Thus the
principal part is diagonal in the target index \(\alpha\), while all
coupling between the components occurs in the first-order semilinear
term \(F\).  We write \(F=(F^1,\ldots,F^{n})\).

Because \(\overline{\mathcal U_R}\) and
\(\overline{\mathcal V_R}\) are compact coordinate sets, there are
constants \(0<\lambda_R\leq\Lambda_R<\infty\) such that
\[
 \lambda_R|\xi|^2
 \leq g^{ab}(x)\xi_a\xi_b
 \leq \Lambda_R|\xi|^2
\]
for every \(x\in\mathcal U_R\) and \(\xi\in\mathbb R^{n}\).  Moreover,
all coordinate derivatives of \(g^{ab}\), \(\Gamma^c_{ab}\), and
\(\overline\Gamma^\alpha_{\beta\gamma}\) are uniformly bounded when
\(x\in\overline{\mathcal U_R}\) and
\(z\in\overline{\mathcal V_R}\).  Together with
\eqref{eq:Us-uniform-coordinate-W1infty}, this implies
\begin{equation}\label{eq:harmonic-system-RHS-Linfty}
 \sup_j
 \bigl\|F(\,\cdot,u_j,Du_j)\bigr\|_{L^\infty(\Omega_\ell^{(2)})}
 \leq C(R)
 \qquad(1\leq\ell\leq N_R).
\end{equation}
Indeed, \(z=u_j(x)\) ranges in the fixed compact set \(K_R\), while
\(P=Du_j(x)\) ranges in a fixed bounded subset of
\(\mathbb R^{n\times n}\).

Using \eqref{eq:harmonic-system-RHS-Linfty}, we may now apply
the scalar interior estimates componentwise, because
\eqref{eq:harmonic_system_local} has the same uniformly elliptic scalar
principal operator \(g^{ab}(x)\partial_{ab}\) in every component.  For
every \(1<p<\infty\), the interior \(W^{2,p}\) estimate
\cite[Theorem~9.11 and its proof, pp.~235--236]{GilbargTrudinger} gives
\begin{align}
 \|u_j\|_{W^{2,p}(\Omega_\ell^{(1)})}
 &\leq C(R,p)
 \left(
 \|u_j\|_{L^p(\Omega_\ell^{(2)})}
 +\|F(\,\cdot,u_j,Du_j)\|_{L^p(\Omega_\ell^{(2)})}
 \right)\nonumber\\
 &\leq C(R,p).
\label{eq:harmonic-system-W2p}
\end{align}
Choose \(p>n\), and put
\[
\theta:=1-\frac{n}{p}\in(0,1).
\]

By \eqref{eq:harmonic-system-W2p} and Sobolev embedding
\cite[Theorem~7.26, pp.~171--172]{GilbargTrudinger}, the maps \(u_j\)
are uniformly bounded in \(C^{1,\theta}(\Omega_\ell^{(1)})\).
Since the function \(F(x,z,P)\) in
\eqref{eq:harmonic-system-nonlinearity} is smooth on every compact set in
its variables, it follows that
\[
 \bigl\|F(\,\cdot,u_j,Du_j)\bigr\|_{C^{0,\theta}(\Omega_\ell^{(1)})}
 \leq C(R,p).
\]
The interior Schauder estimate
\cite[Theorem~6.2 and its proof, pp.~91--92]{GilbargTrudinger}, applied
componentwise to \eqref{eq:harmonic_system_local}, now gives
\[
 \|u_j\|_{C^{2,\theta}(\Omega_\ell^{(0)})}
 \leq C(R,p).
\]
More generally, smoothness of \(F\) gives the implication
\[
 u_j\text{ uniformly bounded in }C^{r+1,\theta}
 \quad\Longrightarrow\quad
 F(\,\cdot,u_j,Du_j)\text{ uniformly bounded in }C^{r,\theta}.
\]
Applying the Schauder estimate to the original equation on successively
smaller coordinate balls, and inducting on \(r\), therefore yields, for
every integer \(k\geq0\),
\begin{equation}\label{eq:all_Ck_bounds}
 \|u_{s_j}\|_{C^k(B(o,R))}\leq C(R,k).
\end{equation}
The constants are independent of \(j\), because only finitely many source
balls are used and all source coefficients, target coefficients, image
values, and first derivatives range in fixed compact sets.

The bounds \eqref{eq:all_Ck_bounds}, Arzel\`a--Ascoli, and a
diagonal argument over \(R=1,2,\ldots\) produce
a subsequence converging in \(C^\infty_{\mathrm{loc}}\) to a smooth map
\(V\).  Since each compact source ball is mapped into a compact subset of
\(\Hn\), the limit takes values in \(\Hn\).  Passing to the limit in
\eqref{eq:harmonic_system_local} on each coordinate ball shows that
\(V\) is harmonic.

\textbf{Step (3)}. The common global Lipschitz inequality passes to
the limit on every pair of points, so \(V\) is globally Lipschitz.
Likewise, if
\[
L^{-1}d(x,x')-C
\leq d(u_s(x),u_s(x'))
\leq Ld(x,x')+C,
\]
then the same inequalities hold for \(V\).

It remains only to check coarse surjectivity.  Let \(C_0\) be the common
coarse-surjectivity constant and fix \(z\in\Hn\).  Choose \(x_j\) with
\[
d(u_{s_j}(x_j),z)\leq C_0.
\]
The convergence of \(u_{s_j}(o)\) implies that
\(d(u_{s_j}(o),z)\) is bounded.  

Hence \(d(u_{s_j}(o),u_{s_j}(x_j))\) is bounded, and the common lower coarse
inequality bounds \(d(o,x_j)\).  

A subsequence of \(x_j\) converges to
some \(x\), and smooth local convergence gives
\[
d(V(x),z)\leq C_0.
\]

Since \(z\) was arbitrary, \(V\) is coarsely onto and therefore a rough
isometry.
\end{proof}

%\ReviewUnusedLabel{notation:Neighborhood-of-ray}

\begin{lemma}\label{lem:boundary_V}
Every limit \(V\) furnished by \Cref{lem:Us_compactness} has boundary map
\[
\bdry V=A
\qquad\text{on }\R^{n-1}\cup\{\infty\}.
\]
\end{lemma}

\begin{proof}
\textbf{Step (1)}. Fix \(\zeta\in\R^{n-1}\cup\{\infty\}\), and let
\(\alpha:[0,\infty)\to\Hn\) be the unit-speed ray from \(o\) to
\(\zeta\).  The curve \(u_{s_j}\circ\alpha\) is a quasi-geodesic ray with constants independent of \(j\), starts at \(u_{s_j}(o)\), and has ideal endpoint \(\phi_{s_j}(\zeta)\).  

Let \(\beta_j\) be the unit-speed
geodesic ray from \(u_{s_j}(o)\) to that endpoint \(\phi_{s_j}(\zeta)\).

By the ray statement in \Cref{lem:Morse}, there is
\(R_1<\infty\), independent of \(j\), such that
\begin{equation}\label{eq:morse_rays}
d\bigl(u_{s_j}(\alpha(T)),\beta_j([0,\infty))\bigr)\leq R_1
\qquad(T\geq0).
\end{equation}

For finite \(\zeta\), \eqref{eq:fs_to_A} gives
\(\phi_{s_j}(\zeta)\to A\zeta\); for \(\zeta=\infty\), both sides are
\(\infty\).  Also \(u_{s_j}(o)\to V(o)\).  Hence
\Cref{lem:visual_facts}(\ref{item:moving_rays}) implies that
\(\beta_j\to\beta\) in \(C^\infty\) on compact parameter intervals,
where \(\beta\) is the ray from \(V(o)\) to \(A\zeta\).

Fix \(T\geq0\).  By \eqref{eq:morse_rays}, choose \(t_j(T)\geq0\) with
\[
d(u_{s_j}(\alpha(T)),\beta_j(t_j(T)))\leq R_1.
\]
Let \(L,C\) be common rough-isometry constants.  Since
\(d(\alpha(0),\alpha(T))=T\),
\begin{equation}\label{eq:tj_bounds}
L^{-1}T-C-R_1
\leq t_j(T)
\leq LT+C+R_1.
\end{equation}
Indeed, \(t_j(T)=d(u_{s_j}(o),\beta_j(t_j(T)))\), and the two inequalities
follow from the triangle inequality and the two coarse bounds for
\(u_{s_j}\).

\textbf{Step (2)}. By \eqref{eq:tj_bounds}, for each fixed
\(T\), the parameters \(t_j(T)\) lie in a compact interval.  After passing to a subsequence that may depend on \(T\), we may assume that \(t_j(T)\to t(T)\).

Local convergence of
\(u_{s_j}\), compact-parameter convergence of \(\beta_j\), and the
preceding distance estimate show that, for every \(T\), there exists such a
number with
\[
d(V(\alpha(T)),\beta(t(T)))\leq R_1,
\qquad
L^{-1}T-C-R_1\leq t(T)\leq LT+C+R_1.
\]
The conclusion for a given \(T\) is independent of the auxiliary
subsequence.  Select one admissible \(t(T)\) for each \(T\).  Then
\[
t(T)\geq L^{-1}T-C-R_1\longrightarrow\infty,
\]
so \(\beta(t(T))\to A\zeta\), and
\Cref{lem:visual_facts}(\ref{item:bounded_same_endpoint}) gives
\[
V(\alpha(T))\longrightarrow A\zeta
\qquad(T\to\infty).
\]
Because \(V\) is a rough isometry, it has a boundary extension; the last
identity identifies that extension at every \(\zeta\), proving
\(\bdry V=A\).
\end{proof}

\section{Identification of the moving-ball asymptotic}
%\ReviewUnusedLabel{sec:moving-ball}

The preceding section established that every sequence
\(s_j\downarrow0\) has a subsequence for which \(u_{s_j}\) converges
smoothly on compact subsets to a harmonic rough isometry
\(V:\Hn\to\Hn\), and that every such limit satisfies
\begin{align}
\bdry V=A.
\nonumber
\end{align}
It remains to identify \(V\).

We first construct the explicit harmonic model \(H_A\).  Li--Wang
uniqueness then identifies every subsequential limit with \(H_A\), which
upgrades subsequential compactness to convergence of the entire
renormalized family.  This produces the moving-ball asymptotic.

Let \(\sigma_1(A),\dots,\sigma_{n-1}(A)\) denote the Euclidean singular
values of \(A\), in the sense fixed in the Introduction, and let \(A^*\)
denote the Euclidean adjoint of \(A\).

Then all \(\sigma_i(A)\) are positive, and
\begin{align}
\sum_{i=1}^{n-1}\sigma_i(A)^2=\operatorname{tr}(A^*A)=\|A\|_{\HS}^2. \nonumber
\end{align}

We set
\begin{align}
&c_A:=\left(\frac{\tr(A^*A)}{n-1}\right)^{1/2}=\left(\frac{\|A\|_{\HS}^2}{n-1}\right)^{1/2}\label{eq:cA_def}\\
&H_A(x,y):=(Ax,c_Ay): \mathbb{H}^{n}\rightarrow\mathbb{H}^{n}.\label{eq:HA_def}
\end{align}
%\SuggestedDelete{\textup{\scriptsize Delete the stray standalone \texttt{\textbackslash}.}}

\begin{lemma}\label{lem:HA_harmonic}
The map \(H_A:\Hn\to\Hn\) is a harmonic bi-Lipschitz diffeomorphism with
\(\bdry H_A=A\).  At every \(p\in\mathbb H^{n}\), the singular values of $d(H_A)_p:
\bigl(T_p\mathbb H^{n},g_{\mathbb H}\bigr)
\longrightarrow
\bigl(T_{H_A(p)}\mathbb H^{n},g_{\mathbb H}\bigr)$ are
\begin{equation}\label{eq:HA_singular_values}
\frac{\sigma_1(A)}{c_A},\dots,
\frac{\sigma_{n-1}(A)}{c_A},1,
\end{equation}
and therefore $|dH_A|^2\equiv n.$
\end{lemma}

\begin{proof}
\textbf{Step (1)}. Write target coordinates as
\((Y^1,\dots,Y^{n-1},Y)\), where \(Y=c_Ay\).  For the upper half-space
metric, the nonzero target Christoffel symbols are
\[
\overline\Gamma^a_{bY}
=
\overline\Gamma^a_{Yb}
=-\frac{\delta^a_b}{Y},
\qquad
\overline\Gamma^Y_{ab}=\frac{\delta_{ab}}Y,
\qquad
\overline\Gamma^Y_{YY}=-\frac1Y.
\]
The scalar hyperbolic Laplacian in dimension \(n\) is
\[
\Delta_{\mathbb H}
=
y^2\left(\sum_{a=1}^{n-1}\partial_{x_ax_a}+\partial_{yy}\right)
-(n-2)y\partial_y.
\]

For a horizontal component \(H_A^a=(Ax)^a\),
\(\Delta_{\mathbb H}H_A^a=0\).  

Every possible connection term is mixed
horizontal--vertical.  It vanishes because the inverse domain metric is
diagonal, \(\partial_yH_A^a=0\), and \(\partial_{x_b}Y=0\).  

Recall that if \(F:(M,g)\to(N,h)\) is smooth, then its tension field is
\begin{align}
\tension(F):=\operatorname{tr}_g(\nabla dF).
\nonumber
\end{align}
In source coordinates \(z^i\) and target coordinates \(Y^\alpha\), its
components are
\begin{align}
\tension^\alpha(F)
&=
\Delta_gF^\alpha
+
g^{ij}\overline\Gamma^\alpha_{\beta\gamma}(F)
\partial_iF^\beta\partial_jF^\gamma,
\nonumber
\end{align}
where \(\overline\Gamma^\alpha_{\beta\gamma}\) are the Christoffel symbols
of the target metric.  The map \(F\) is harmonic if and only if
\(\tension^\alpha(F)=0\) for every \(\alpha\).

Since the only target Christoffel symbols
with upper horizontal index \(a\) are
\(\overline\Gamma^a_{bY}=\overline\Gamma^a_{Yb}=-\delta_b^a/Y\), we obtain
\begin{align}
\tension^a(H_A)
&=
\Delta_{\mathbb H}H_A^a
-\frac{2}{Y}g^{ij}\partial_iH_A^a\partial_jY.
\nonumber
\end{align}
The second term vanishes because \(H_A^a\) depends only on the horizontal
variables, \(Y=c_Ay\) depends only on \(y\), and the inverse source metric
is diagonal.  

Therefore
\begin{align}
\tension^a(H_A)=0
\qquad(1\leq a\leq n-1).
\nonumber
\end{align}

For the vertical component,
\[
\Delta_{\mathbb H}(c_Ay)=-(n-2)c_Ay.
\]
The horizontal connection contribution is
\[
\begin{aligned}
g^{ij}\overline\Gamma^Y_{bc}\partial_iH_A^b\partial_jH_A^c&=y^2\sum_{i=1}^{n-1}\frac1{c_Ay}
\sum_{b=1}^{n-1}(A_i^b)^2=\frac{y}{c_A}\|A\|_{\HS}^2,
\end{aligned}
\]
and the vertical connection contribution is
\[
g^{yy}\overline\Gamma^Y_{YY}(\partial_yY)^2
=
y^2\left(-\frac1{c_Ay}\right)c_A^2
=-c_Ay.
\]
Thus
\[
\tension^Y(H_A)
=
y\left(-(n-1)c_A+\frac{\|A\|_{\HS}^2}{c_A}\right)=0
\]
by \((n-1)c_A^2=\|A\|_{\HS}^2\).  Hence \(H_A\) is harmonic.

\textbf{Step (2)}. At \((x,y)\), use the source hyperbolic orthonormal frame
\[
y\partial_{x_1},\dots,y\partial_{x_{n-1}},y\partial_y
\]
and the target frame
\[
c_Ay\partial_{Y^1},\dots,c_Ay\partial_{Y^{n-1}},c_Ay\partial_Y.
\]
In these frames,
\begin{equation}\label{eq:HA_matrix}
dH_A=
\begin{pmatrix}
A/c_A&0\\
0&1
\end{pmatrix}.
\end{equation}
The matrix formula \eqref{eq:HA_matrix} proves
\eqref{eq:HA_singular_values}.

Since \(A\) is invertible,
all singular values are bounded above and below by positive constants, so
\(H_A\) is a global hyperbolic bi-Lipschitz diffeomorphism.  Its extension
to \(\R^{n-1}\cup\{\infty\}\) is \(x\mapsto Ax\) and
\(\infty\mapsto\infty\), proving \(\bdry H_A=A\).
Finally, by the Hilbert--Schmidt convention
\eqref{eq:HS_norm_def},
\[
|dH_A|^2
=
\sum_{i=1}^{n-1}\frac{\sigma_i(A)^2}{c_A^2}+1
=
\frac{\tr(A^*A)}{\tr(A^*A)/(n-1)}+1
=n.
\]
\end{proof}

Li--Wang formulate the notion of a rough isometry using a coarse inverse
rather than coarse surjectivity.  Under the two-sided bounds
\eqref{eq:general-rough-isometry-bounds}, their formulation is equivalent to
Definition~\ref{def:rough-isometry}.  Indeed, suppose that
\(F:X\to Y\) is
\(C_0\)-coarsely surjective.  For each \(y\in Y\), choose \(G(y)\in X\)
such that
\begin{align}
d_Y(F(G(y)),y)\leq C_0.
\nonumber
\end{align}
Taking \(y=F(x)\) and applying the lower bound in
\eqref{eq:general-rough-isometry-bounds} to \(G(F(x))\) and \(x\), we obtain
\begin{align}
L^{-1}d_X(G(F(x)),x)-C
&\leq d_Y(F(G(F(x))),F(x))
\leq C_0.
\nonumber
\end{align}
Consequently,
\begin{align}
d_X(G(F(x)),x)\leq L(C+C_0)
\qquad\text{for every }x\in X.
\nonumber
\end{align}
Thus \(G\) is a coarse inverse of \(F\).  Conversely, the existence of a
coarse inverse immediately implies coarse surjectivity.  Hence Li--Wang's
theorem may be stated using the terminology of
Definition~\ref{def:rough-isometry}.

The following is Li--Wang \cite[Theorem~2.3, pp.~433--436]{LiWang}.
\begin{theorem}[Li--Wang's harmonic rough-isometry uniqueness theorem]
\label{thm:LiWang_uniqueness}
Let \(X\) be a Cartan--Hadamard manifold and let \(Y\) be a Hadamard space
whose curvature satisfies
\begin{align}
K_Y\leq-a^2<0
\nonumber
\end{align}
in the metric triangle-comparison sense.  Suppose that the isometry groups
of \(X\) and \(Y\) act cocompactly.  If
\begin{align}
F,G:X\longrightarrow Y
\nonumber
\end{align}
are harmonic rough isometries inducing the same boundary map at infinity,
then \(F=G\).
\end{theorem}

For the application below, one further compatibility point remains.
Because the target is the Hadamard manifold \(\Hn\), a smooth classically
harmonic map is locally energy minimizing; hence the smooth maps used below
are harmonic in the variational sense used by Li--Wang.

\begin{proposition}\label{prop:tangent}
In the normalized setting of \Cref{sec:tangent}, let
\(r(t)=(\mathbf{0},e^{-t})\) and
\(A=D\phi_{\mathbf{0}}
\in\operatorname{GL}(n-1,\R)\).  Then, 
\begin{equation}\label{eq:sharp_tube_limit}
\lim_{t\to\infty}\sup_{z\in B(r(t),R)}\bigl||du(z)|^2-n\bigr|=0, \quad \quad \forall R\geq 0.
\end{equation}
\end{proposition}

\begin{proof}
\textbf{Step (1)}. Recall the invertible map
\(A=D\phi_{\mathbf{0}}\) and the definitions of \(c_A\) and \(H_A\)
in \eqref{eq:cA_def}--\eqref{eq:HA_def}.

Let \(s_j\downarrow0\).  By \Cref{lem:Us_compactness}, a subsequence $u_{s_j}$ converges in \(C^\infty_{\mathrm{loc}}\) to a harmonic rough isometry
\(V\).  By \Cref{lem:boundary_V,lem:HA_harmonic},
\[
\bdry V=A=\bdry H_A.
\]

We now verify the hypotheses of \Cref{thm:LiWang_uniqueness}.  Real
hyperbolic space is a Cartan--Hadamard manifold of curvature \(-1\), and
its full isometry group acts transitively and hence cocompactly.  Both
\(V\) and \(H_A\) are harmonic rough isometries, and
\(\bdry V=A=\bdry H_A\).  Thus \Cref{thm:LiWang_uniqueness} applies and
gives
\begin{equation}\label{eq:V_equals_HA}
V=H_A.
\end{equation}

By \eqref{eq:V_equals_HA}, every sequence \(s_j\downarrow0\)
has a subsequence converging to the same limit \(H_A\).

Hence the boundary-tangent renormalized maps \(u_s\) defined in
\eqref{eq:Us_def} satisfy
\begin{equation}\label{eq:tangent_full_statement}
u_s\longrightarrow H_A
\qquad\text{in }C^\infty_{\mathrm{loc}}(\Hn,\Hn)
\quad(s\downarrow0).
\end{equation}

\textbf{Step (2)}. Fix \(R<\infty\).  The \(C^1\)-part of
\eqref{eq:tangent_full_statement} and the identity \(|dH_A|^2\equiv n\) from Lemma~\ref{lem:HA_harmonic} give
\begin{equation}\label{eq:energy_uniform_fixed_ball}
\lim_{s\rightarrow 0}\sup_{w\in B(o,R)}
\left||du_s(w)|^2-n\right|= 0.
\end{equation}
Set \(s=e^{-t}\).  By \eqref{eq:vertical_ray},
\(r(t)=\delta_s(o)\), and because \(\delta_s\) is an isometry,
\[
\delta_s(B(o,R))=B(r(t),R).
\]

The source and target dilations in
\(u_s=\delta_{1/s}\circ u\circ\delta_s\) are isometries, so
\begin{equation}\label{eq:energy_scaling_identity}
|du_s(w)|^2=|du(\delta_sw)|^2
\qquad(w\in\Hn).
\end{equation}

Combining \eqref{eq:energy_uniform_fixed_ball} and
\eqref{eq:energy_scaling_identity} gives
\begin{align}
\lim_{t\to\infty}
\sup_{z\in B(r(t),R)}
\bigl||du(z)|^2-n\bigr|
&=
\lim_{s\downarrow0}
\sup_{w\in B(o,R)}
\bigl||du_s(w)|^2-n\bigr|=0.
\nonumber
\end{align}
This proves \eqref{eq:sharp_tube_limit}.  
\end{proof}

\begin{remark}
The preceding argument may be viewed as a tangent-map version of the
\(C^1\) boundary regularity theory of Li--Tam
\cite{LiTam,LiTamII}.  

In that theory, one assumes \(C^1\) control of a
proper harmonic map up to the geometric boundary together with a
nonvanishing differential of its boundary map.  The harmonic-map equation
then determines the asymptotic normal scale from the tangential boundary
energy.  An explicit formulation of this boundary calculation in the more
general asymptotically hyperbolic setting is given in
\cite[Lemma~1.3 and Corollary~1.4]{AkutagawaMatsumoto}. 

Translated into
the fixed upper-half-space normalization used here, for a source of
dimension \(n\) the
conclusion is
\begin{align}
|du|^2\longrightarrow n
\nonumber
\end{align}
as the interior point approaches the ideal boundary. This is an asymptotic interior statement; \(|du|^2\) is not literally
evaluated at a point of the ideal boundary.

The point of \Cref{prop:tangent} is that no \(C^1\) compactification of
\(u\), and no \(C^1\) regularity of \(\phi=\bdry u\) on a boundary
neighborhood, is required.  

For the globally Lipschitz harmonic rough isometry considered here, one
invertible boundary differential \(A=D\phi_{\mathbf{0}}\) is enough.  This
is a differential of the boundary map, not a differential of the interior
map \(u\) at an ideal point.  The boundary-tangent renormalization and
\Cref{thm:LiWang_uniqueness} identify the model \(H_A\) defined in
\eqref{eq:cA_def}--\eqref{eq:HA_def}.  Although \(H_A\) need not be an
isometry, \Cref{lem:HA_harmonic} gives \(|dH_A|^2\equiv n\); invertibility
of \(A\) makes \(H_A\) a bi-Lipschitz rough isometry and permits the
uniqueness theorem to identify every subsequential limit.

Consequently, \Cref{prop:tangent} yields the moving-ball asymptotic, and
\Cref{lem:cocompact_propagation}, together with \(\Gamma\)-invariance,
converts that asymptotic into the global identity \(|du|^2\equiv n\).
\end{remark}

% ================================================================
\section{The sharp estimate and Mostow rigidity}
%\ReviewUnusedLabel{sec:sharp}
% ================================================================

The following elementary lemma is the deterministic propagation step.

\begin{lemma}\label{lem:cocompact_propagation}
Let a group \(G\) act cocompactly by isometries on a metric space \(X\),
and let $\mathfrak{q}:X\longrightarrow\R$ be \(G\)-invariant.  Suppose that there exist a curve
\(\gamma: [0,\infty)\rightarrow X\) and a constant \(c\in\R\) such that 
\begin{equation}\label{eq:abstract_moving_ball_limit}
\lim_{t\to\infty}
\sup_{z\in B(\gamma(t),R)}|\mathfrak{q}(z)-c|=0, \quad \quad \text{for every }R\geq0.
\end{equation}
Then $\mathfrak{q}\equiv c$ on $X$.
\end{lemma}

\begin{proof}
Cocompactness provides a nonempty compact set \(K\subset X\) such that
\[
GK=X.
\]

Set $D:=\operatorname{diam}K<\infty$.  For every \(t\geq 0\), choose \(g_t\in G\) and \(k_t\in K\) such that
\[
\gamma(t)=g_tk_t.
\]

Note that $g_t$ is an isometry, hence
\[
d(g_tx,\gamma(t))
=
d(g_tx,g_tk_t)
=
d(x,k_t)
\leq D, \quad \quad \quad \forall x\in K.
\]
Hence $g_tK\subset B(\gamma(t),D)$.

Using the \(G\)-invariance of \(\mathfrak{q}\),
\begin{align*}
\sup_{x\in K}|\mathfrak{q}(x)-c|&=\sup_{x\in K}|\mathfrak{q}(g_tx)-c|=
\sup_{y\in g_tK}|\mathfrak{q}(y)-c|\leq
\sup_{z\in B(\gamma(t),D)}|\mathfrak{q}(z)-c|.
\end{align*}

Letting \(t\to\infty\) and applying
\eqref{eq:abstract_moving_ball_limit} with \(R=D\) gives
\[
\sup_{x\in K}|\mathfrak{q}(x)-c|=0.
\]

Thus \(\mathfrak{q}=c\) on \(K\).  Since \(GK=X\) and \(\mathfrak{q}\) is invariant, the identity extends to \(X\).
\end{proof}

\begin{proof}[Proof of \Cref{thm:main}]
\textbf{Step (1)}. All hypotheses of Proposition \ref{prop:tangent} now hold:
\(u\) is harmonic, globally Lipschitz, and a rough isometry, while
\(D\phi_\xi\) exists and is invertible because \(\xi\in E_\phi\).  Applying Proposition \ref{prop:tangent} gives
\begin{equation}\label{eq:q_tube_limit}
\lim_{t\to\infty}
\sup_{z\in B(r(t),R)}\bigl||du(z)|^2-n\bigr|=0, \quad \quad \forall R\geq 0.
\end{equation}

\textbf{Step (2)}. Differentiating equivariance gives
\[
du_{\gamma x}\circ d\gamma_x
=
d\rho(\gamma)_{u(x)}\circ du_x, \quad \quad \forall \gamma\in \Gamma.
\]

Since the differentials of the source and target deck transformations are linear isometries, the Hilbert--Schmidt norm is preserved and
\begin{equation}\label{eq:q_invariant}
|du(\gamma x)|^2= |du(x)|^2 \qquad (\gamma\in\Gamma,\ x\in\Hn).
\end{equation}

Thus the deck group \(\Gamma\) acts cocompactly by isometries on \(\Hn\), and \(|du|^2\) is \(\Gamma\)-invariant by \eqref{eq:q_invariant}.  

From (\ref{eq:q_tube_limit}) and \Cref{lem:cocompact_propagation}, we get $|du|^2\equiv n$ on $\Hn$.

Since \(|du|^2\) descends to \(|d\bar u|^2\), the conclusion follows.
\end{proof}

In the rest of this section, \(\Jac F\) denotes the signed Jacobian determined by the chosen
orientations, whereas \( |\Jac F| \) is the product of the singular
values of \(dF\).

\begin{proof}[Proof of \Cref{cor:mostow}]
\textbf{Step (1)}. Theorem~\ref{thm:main} gives
\begin{align}
|d\bar u|^2\equiv n.\label{gradient-u-equality}
\end{align}

Fix \(x\in M\), and let
\[
\lambda_1(x),\dots,\lambda_n(x)
\]
be the singular values of \(d\bar u_x\).  

The arithmetic--geometric mean inequality gives
\begin{align}
|\Jac\bar u(x)|^{2/n}
=
\left(\prod_{i=1}^n\lambda_i(x)^2\right)^{1/n}
\leq
\frac1n\sum_{i=1}^n\lambda_i(x)^2
= \frac{1}{n}|d\bar u|^2(x)= 1.
\label{eq:jac_bounds_both}
\end{align}

Let
\begin{align}
\bar v_h:N\longrightarrow M
\nonumber
\end{align}
be the Eells--Sampson harmonic representative of the homotopy class of the
previously chosen homotopy inverse \(\bar v\).  After the corresponding
choice of lifts, its lift is \(\rho^{-1}\)-equivariant.  All the hypotheses
used above are symmetric under interchanging
\((M,\Gamma,\rho)\) with \((N,\Lambda,\rho^{-1})\).  Applying
Theorem~\ref{thm:main} to \(\bar v_h\) therefore gives
\begin{align}
|d\bar v_h|^2\equiv n,
\qquad
|\Jac\bar v_h|\leq1.
\nonumber
\end{align}

Because \(\bar u\) lies in a homotopy-equivalence class,
\(|\deg\bar u|=1\): if \(h:N\to M\) is a homotopy inverse, then
\(\deg(h)\deg(\bar u)=\deg(h\circ\bar u)=\deg(\Id_M)=1\), and both
degrees are integers.  The degree formula gives
\[
\int_M\Jac(\bar u)\,\dd\operatorname{vol}_M
=
\deg(\bar u)\Vol(N).
\]
Therefore
\begin{equation}\label{eq:volume_forward}
\Vol(N)
=
\left|\int_M\Jac(\bar u)\,\dd\operatorname{vol}_M\right|
\leq
\int_M|\Jac\bar u|\,\dd\operatorname{vol}_M
\leq
\Vol(M).
\end{equation}

Applying the same argument to \(\bar v_h\) gives
\begin{equation}\label{eq:volume_reverse}
\Vol(M)\leq\Vol(N).
\end{equation}

Combining \eqref{eq:volume_forward} and
\eqref{eq:volume_reverse} gives $\Vol(M)=\Vol(N)$.

The first and last terms in \eqref{eq:volume_forward} are now equal.
Consequently,
\[
\int_M|\Jac\bar u|\,\dd\operatorname{vol}_M=\Vol(M).
\]

By \eqref{eq:jac_bounds_both}, we have 
\(0\leq|\Jac\bar u|\leq1\). Note
\[
\int_M\bigl(1-|\Jac\bar u|\bigr)
\,\dd\operatorname{vol}_M=0.
\]
The integrand is continuous and nonnegative, so it vanishes identically:
\begin{align}
|\Jac\bar u|\equiv1. \label{Jac-u-equality}
\end{align}

\textbf{Step (2)}. By (\ref{gradient-u-equality}) and (\ref{Jac-u-equality}), 
\begin{equation}\label{eq:trace_det_equalities}
\sum_{i=1}^n\lambda_i(x)^2=n,
\qquad
\prod_{i=1}^n\lambda_i(x)=1.
\end{equation}

By \eqref{eq:trace_det_equalities}, equality holds in the
arithmetic--geometric mean inequality
\[
1
=
\left(\prod_{i=1}^n\lambda_i(x)^2\right)^{1/n}
\leq
\frac1n\sum_{i=1}^n\lambda_i(x)^2
=
1,
\]
so $\lambda_1(x)=\cdots=\lambda_n(x)=1.$

Thus \(d\bar u_x\) is a linear isometry for every \(x\in M\).  Hence
\(\bar u\) is a local Riemannian isometry and, in particular, a local
diffeomorphism.

Because \(M\) is compact, \(\bar u\) is proper.  Its image is open because
\(\bar u\) is a local diffeomorphism and closed because it is the continuous
image of a compact space.  Since \(N\) is connected, \(\bar u\) is
surjective.  A proper local diffeomorphism is a covering map.  

The subgroup of \(\pi_1(N)\) determined by this covering is
\(\bar u_*(\pi_1(M))\).  But \(\bar u\) is homotopic to the original
homotopy equivalence, so \(\bar u_*\) is the prescribed isomorphism
\(\rho\).  The covering therefore has one sheet.  Hence \(\bar u\) is a
diffeomorphism and, being a local isometry, is a global Riemannian isometry.
\end{proof}

% ================================================================
%================================================================

% ================================================================

\end{document}